\documentclass[
  bibliography=totoc,
]{scrartcl}

\usepackage[utf8]{inputenc}
\usepackage[T1]{fontenc}

\setkomafont{disposition}{\normalfont\bfseries}

\usepackage{microtype}

\usepackage[inline]{enumitem}

\usepackage[final]{hyperref}
\hypersetup{
  colorlinks=true,
  allcolors=black}
\usepackage[numbered]{bookmark}  %

\usepackage{natbib}
\usepackage{doi}

\usepackage{subcaption}

\usepackage[final]{graphicx}

\graphicspath{%
  {figures/}%
}

\usepackage{tikz}

\usetikzlibrary{%
  positioning,
}

\usepackage{xcolor}

\definecolor{TUred}{RGB}{165,30,55}
\definecolor{TUgold}{RGB}{180,160,105}
\definecolor{TUdark}{RGB}{50,65,75}
\definecolor{TUgray}{RGB}{175,179,183}

\definecolor{TUdarkblue}{RGB}{65,90,140}
\definecolor{TUblue}{RGB}{0,105,170}
\definecolor{TUlightblue}{RGB}{80,170,200}
\definecolor{TUlightgreen}{RGB}{130,185,160}
\definecolor{TUgreen}{RGB}{125,165,75}
\definecolor{TUdarkgreen}{RGB}{50,110,30}
\definecolor{TUocre}{RGB}{200,80,60}
\definecolor{TUviolet}{RGB}{175,110,150}
\definecolor{TUmauve}{RGB}{180,160,150}
\definecolor{TUbeige}{RGB}{215,180,105}
\definecolor{TUorange}{RGB}{210,150,0}
\definecolor{TUbrown}{RGB}{145,105,70}

\newcommand*{\figlegendline}[1]{(\protect\tikz[baseline]{\protect\draw[#1, very thick] (0,.6ex) -- (1em,.6ex);})}

\usepackage{amsmath}
\usepackage{mathtools}
\usepackage{physics}

\usepackage{amssymb}
\usepackage{bm}
\usepackage{nicefrac}
\usepackage{bbm}

\usepackage{amsthm}
\usepackage{thmtools}

\usepackage{quiver}

\usepackage{etoolbox}
\usepackage{xparse}

\allowdisplaybreaks    %

\numberwithin{equation}{section}  %

\newcommand*{\defeq}{\coloneqq}  %
\providecommand{\where}{}

\DeclarePairedDelimiterX{\set}[1]\{\}{%
\renewcommand*{\where}{\colon}
#1}

\newcommand*{\setsym}[1]{\ensuremath{\mathbb{#1}}}

\newcommand*{\N}{\setsym{N}}

\newcommand*{\R}{\setsym{R}}

\newcommand*{\preim}[1]{#1^{-1}}  %

\DeclarePairedDelimiterX{\restrict}[3]{.}{\rvert_{#2}^{#3}}{#1}  %

\newcommand*{\inv}{^{-1}}

\newcommand*{\openarg}[1][\cdot]{\,#1\,}

\newcommand*{\closure}[1]{\overline{#1}}

\NewDocumentCommand{\range}{s m}{%
  \operatorname{ran}
  \IfBooleanTF{#1}{\left(}{(}
  #2
  \IfBooleanTF{#1}{\right)}{)}
}

\NewDocumentCommand{\kernel}{s m}{%
  \operatorname{ker}
  \IfBooleanTF{#1}{\left(}{(}
  #2
  \IfBooleanTF{#1}{\right)}{)}
}

\newcommand*{\oball}[2]{B_{#1} \left( #2 \right)}  %
\newcommand*{\dualop}{'}

\makeatletter
\DeclarePairedDelimiterXPP{\@inprod}[3]{}{\langle}{\rangle}{\ifstrempty{#3}{}{_{#3}}}{#1, #2}
\NewDocumentCommand{\inprod}{s O{} m m O{}}{%
  \IfBooleanTF{#1}{%
    \@inprod*{#3}{#4}{#5}%
  }{%
    \@inprod[#2]{#3}{#4}{#5}%
  }%
}
\makeatother

\NewDocumentCommand{\cfns}{m o}{%
  C (#1\IfValueT{#2}{, #2})
}

\NewDocumentCommand{\cdfns}{O{0} m o}{%
  C^{#1} (#2\IfValueT{#3}{, #3})
}

\NewDocumentCommand{\holdersp}{m o m}{%
  C^{#1\IfValueT{#2}{, #2}}(\closure{#3})
}

\RenewDocumentCommand{\L}{m o}{%
  \ensuremath{
    L_{#1}
    \IfNoValueF{#2}{\left( #2 \right)}
  }
}

\NewDocumentCommand{\sobolev}{o m o}{%
  \IfNoValueTF{#1}{
    H^{#2}
  }{
    W^{#1,#2}
  }
  \IfNoValueF{#3}{\left( #3 \right)}
}

\NewDocumentCommand{\sobolevtest}{o m o}{%
  \sobolev[#1]{#2}_0
  \IfNoValueF{#3}{\left( #3 \right)}
}

\newcommand*{\linop}[1]{\mathcal{#1}}

\newcommand*{\linfctls}[1]{{\bm{\linop{#1}}}}

\makeatletter
\DeclarePairedDelimiter{\@evallinop@oparg}{(}{)}
\DeclarePairedDelimiter{\@evallinop@fnarg}{(}{)}
\NewDocumentCommand{\evallinop}{m s O{} m s O{} d()}{%
  #1%
  \IfBooleanTF{#2}{%
    \@evallinop@oparg*{#4}%
  }{%
    \@evallinop@oparg[#3]{#4}%
  }%
  \IfValueT{#7}{%
    \IfBooleanTF{#5}{%
      \@evallinop@fnarg*{#7}%
    }{%
      \@evallinop@fnarg[#6]{#7}%
    }%
  }%
}
\makeatother

\NewDocumentCommand{\linopat}{m s O{} m d()}{%
  \IfBooleanTF{#2}{%
    \evallinop{\linop{#1}}*{#4}(#5)%
  }{%
    \evallinop{\linop{#1}}[#3]{#4}(#5)%
  }%
}

\NewDocumentCommand{\linfctlsat}{m s O{} m d()}{%
  \IfBooleanTF{#2}{%
    \evallinop{\linfctls{#1}}*{#4}(#5)%
  }{%
    \evallinop{\linfctls{#1}}[#3]{#4}(#5)%
  }%
}

\NewDocumentCommand{\jacobian}{o m o}{%
  \IfNoValueTF{#1}{%
    \mathrm{D} #2 \IfNoValueF{#3}{\left( #3 \right)}
  }{%
    \IfValueT{#3}{\left.}
    \mathrm{D}_{#1} #2
    \IfValueT{#3}{\right|_{#1 = #3}}
  }
}

\NewDocumentCommand{\hessian}{o m o}{%
  \IfNoValueTF{#1}{%
    \mathrm{H} #2 \IfNoValueF{#3}{\left( #3 \right)}
  }{%
    \IfValueT{#3}{\left.}
    \mathrm{H}_{#1} #2
    \IfValueT{#3}{\right|_{#1 = #3}}
  }
}

\DeclareMathOperator*{\argmin}{arg\,min}

\newcommand*{\sigalg}{\mathcal{A}}
\newcommand*{\borelsigalg}[1]{\mathcal{B} \left( #1 \right)}

\newcommand*{\pushfw}[2]{#2_\star #1}

\makeatletter

\newcommand*{\@given}[1]{%
  \nonscript\:#1\vert%
  \allowbreak%
  \nonscript\:%
  \mathopen{}%
}

\newcommand*{\given}{\@given{}}

\DeclarePairedDelimiterX{\condps}[1]{(}{)}{%
  \renewcommand*{\given}{\@given{\delimsize}}%
  #1%
}

\makeatother

\newcommand*{\pmeas}{\mathrm{P}}  %

\NewDocumentCommand{\prob}{O{\pmeas} s O{} m}{%
  #1%
  \IfBooleanTF{#2}{%
    \condps*{#4}%
  }{%
    \condps[#3]{#4}%
  }%
}

\makeatletter

\DeclarePairedDelimiterX{\@condrv}[1]{.}{.}{%
  \renewcommand*{\given}{\@given{\delimsize}}%
  #1}

\NewDocumentCommand{\condrv}{som}{%
  \IfBooleanTF{#1}{%
    \@condrv*{#3}
  }{%
    \IfNoValueTF{#2}{%
      \begingroup%
      \renewcommand*{\given}{\@given{}}%
      #3%
      \endgroup%
    }{%
      \@condrv[#2]{#3}%
    }
  }
}

\makeatother

\NewDocumentCommand{\expectation}{o o m}{%
  \operatorname{\mathbb{E}}\IfNoValueF{#1}{_{#1\IfNoValueF{#2}{\sim #2}}} \left[ #3 \right]
}
\NewDocumentCommand{\covariance}{o o m m}{%
  \operatorname{Cov}\IfNoValueF{#1}{_{#1\IfNoValueF{#2}{\sim #2}}} \left[ #3, #4 \right]
}
\NewDocumentCommand{\variance}{o o m}{%
  \operatorname{\mathbb{V}}\IfNoValueF{#1}{_{#1\IfNoValueF{#2}{\sim #2}}} \left[ #3 \right]
}

\newcommand*{\gaussian}[2]{{\ensuremath{\operatorname{\mathcal{N}}\left(#1, #2\right)}}}

\NewDocumentCommand{\LkL}{m m o}{%
  #1 #2 \IfValueTF{#3}{#3}{#1}\dualop
}

\declaretheorem[style=plain,numberwithin=section]{theorem}
\declaretheorem[style=plain,sibling=theorem]{proposition}
\declaretheorem[style=plain,sibling=theorem]{lemma}
\declaretheorem[style=plain,sibling=theorem]{corollary}
\declaretheorem[style=definition,sibling=theorem]{definition}
\declaretheorem[style=definition,sibling=theorem]{assumption}
\declaretheorem[style=remark,numberwithin=section]{remark}
\declaretheorem[style=remark,numberwithin=section]{example}

\newcommand*{\latsp}{\mathbb{X}}  %
\newcommand*{\sublatsp}{\mathbb{W}}  %
\newcommand*{\latval}{x}  %
\newcommand*{\sublatval}{w}  %
\newcommand*{\latset}{X}  %
\newcommand*{\sublatset}{W}  %

\newcommand*{\prior}{\mu}  %

\newcommand*{\obsop}{h}  %

\newcommand*{\obssp}{\mathbb{Y}}  %
\newcommand*{\subobssp}{\mathbb{Z}}  %
\newcommand*{\obs}{y}  %
\newcommand*{\subobs}{z}  %
\newcommand*{\obsset}{Y} %
\newcommand*{\subobsset}{Z} %

\newcommand*{\predictive}{\pushfw{\prior}{\obsop}}  %

\newcommand*{\obsfib}[1][\obs]{\obsop\inv \pqty{\set{#1}}}  %

\NewDocumentCommand{\disint}{m m s o D(){\openarg}}{%
  #1^{#2}%
  \IfNoValueF{#4}{%
    \IfBooleanTF{#3}{\condps}{\condps*}%
    {#5 \given #4}%
  }%
}

\newcommand{\priordisint}[1][\obs]{%
  \ifblank{#1}{%
    \disint{\prior}{\obsop}%
  }{%
    \disint{\prior}{\obsop}[#1]%
  }%
}

\newcommand*{\mode}{\latval^\star}

\newcommand*{\omfctl}[1][\prior]{I_{#1}}
\newcommand*{\omdom}[1][\prior]{E_{#1}}  %

\newcommand*{\lebesgue}{\lambda}  %
\newcommand*{\volmeas}[1]{\omega_{#1}}  %

\usepackage[
  obeyDraft,
  textsize=tiny,
  figwidth=0.99\linewidth,
]{todonotes}

\usepackage[
  capitalize,
  nameinlink,
  noabbrev,
]{cleveref}

\makeatletter
\if@cref@capitalise
  \crefname{assumption}{Assumption}{Assumptions}
\else
  \crefname{assumption}{assumption}{assumptions}
\fi
\makeatother

\title{%
  Constructive Disintegration\\
  and Conditional Modes%
}
\author{%
  Nathaël Da Costa\thanks{Tübingen AI Center, University of Tübingen. Equal contribution.}
  \and
  Marvin Pförtner\footnotemark[1]
  \and
  Jon Cockayne\thanks{University of Southampton}
}

\begin{document}

\maketitle

\begin{abstract}
  Conditioning, the central operation in Bayesian statistics, is formalised by the notion of disintegration of measures.
However, due to the implicit nature of their definition, constructing disintegrations is often difficult.
A folklore result in machine learning conflates the construction of a disintegration with the restriction of probability density functions onto the subset of events that are consistent with a given observation.
We provide a comprehensive set of mathematical tools which can be used to construct disintegrations and apply these to find densities of disintegrations on differentiable manifolds.
Using our results, we provide a disturbingly simple example in which the restricted density and the disintegration density drastically disagree.
Motivated by applications in approximate Bayesian inference and Bayesian inverse problems, we further study the modes of disintegrations.
We show that the recently introduced notion of a ``conditional mode'' does not coincide in general with the modes of the conditional measure obtained through disintegration, but rather the modes of the restricted measure.
We also discuss the implications of the discrepancy between the two measures in practice, advocating for the utility of both approaches depending on the modelling context.

\end{abstract}

\tableofcontents

\section{Introduction}
\label{sec:introduction}
The construction of conditional distributions is central to statistics and probabilistic machine learning.
On discrete event spaces, these can often be constructed directly, by applying the definition of conditional probability.
However, this construction breaks down in case the event to be conditioned on has probability zero, which is virtually always the case on continuous event spaces.
In this case, the construction is more involved, and many texts focus solely on the setting where a Lebesgue density is available or, more generally, where the conditional is dominated by, and thus has a density with-respect-to, the base measure.
However, recent works have generated a profusion of examples in which this familiar construction is no longer valid, primarily due to conditionals being defined on a \emph{submanifold} of the support of the measure.
Some examples include:
\begin{itemize}
    \item \textbf{Bayesian deep learning} in which reparameterisation invariances often lead to complex multimodalities in posterior measures (e.g.\ \cite{Wiese2023SymmetryBayesianDL,Laurent2024SymmetryAware}), potentially mitigated by restricting a prior to a submanifold to limit these invariances,
    \item \textbf{Directional statistics} \citep{Mardia1999} where learning problems are defined on manifolds such as the Stiefel manifold of orthonormal frames, a submanifold of $\R^{d \times k}$,
    \item \textbf{Probabilistic numerics} \citep{cockayne_bayesian_2019} which utilises noise-free observations of an unknown, resulting in posteriors concentrated on submanifolds.
\end{itemize}

There is therefore a growing need to be able to characterise conditional distributions on a more intrinsic, measure-theoretic level than by using densities.

In modern probability theory, conditioning is formalised as \emph{disintegration}.
A disintegration of a measure $\prior$ on $\latsp$ along a map $\obsop:\latsp\to\obssp$ is defined as the almost-surely\footnote{Specifically, $\predictive$-almost surely; see \cref{def:disintegration,thm:disintegration}.} unique family of distributions $\set{\priordisint}_{\obs\in\obssp}$ that satisfies the law of total probability.
Existence and almost-sure uniqueness of disintegrations is provided by \emph{disintegration theorems} which have been available under mild conditions since at least \citet{dellacherie_probabilities_1978}; see \citet{chang_conditioning_1997} for a more recent exposition.
However proofs are typically abstract and, crucially, non-constructive outside of conjugate settings, such as with Gaussian $\prior$ and linear $\obsop$.
This paper provides, to our knowledge, the first general set of results for constructing disintegrations.

\begin{figure}[t!]
  \subcaptionbox{%
    \label{fig:mode_example:prior}%
    Prior density and observation fibers \figlegendline{TUgray}
  }[0.49\linewidth][l]{%
    \includegraphics[width=\linewidth]{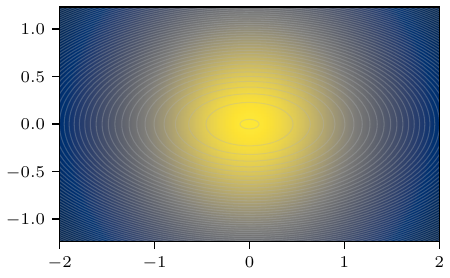}
  }%
  \hfill%
  \subcaptionbox{%
    \label{fig:mode_example:densities}%
    Renormalized restricted density \figlegendline{TUdark} and disintegration density \figlegendline{TUgold} for $y = 1.01$.
  }[0.49\linewidth][l]{%
    \includegraphics[width=\linewidth]{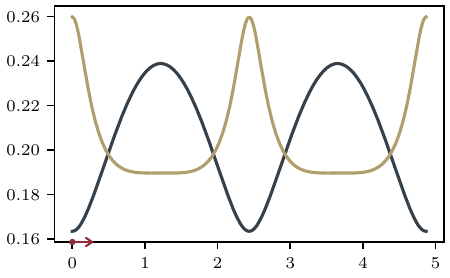}
  }
  \subcaptionbox{%
    \label{fig:mode_example:restriction}%
    Renormalized restricted densities
  }[0.499\linewidth][l]{%
    \includegraphics[width=\linewidth]{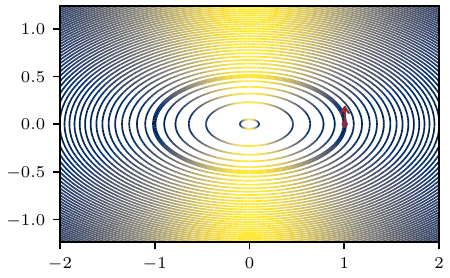}
  }%
  \subcaptionbox{%
    \label{fig:mode_example:disintegration}%
    Disintegration densities
  }[0.499\linewidth][l]{%
    \includegraphics[width=\linewidth]{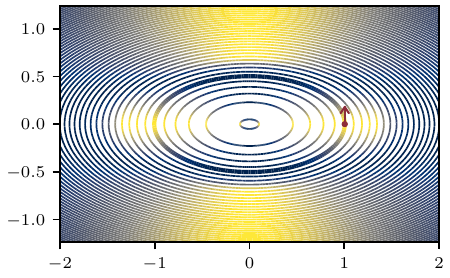}
  }
  \caption{%
    Conditioning a standard Gaussian prior $\prior$ on observations made through the quadratic observation operator $\obsop(\latval) = \frac{\latval_1^2}{a^2} + \frac{\latval_2^2}{b^2}$ with $a = 1$ and $b = \frac{1}{2}$.
    In machine learning folklore, the corresponding conditional distribution is constructed by first restricting the prior's density $\dv{\prior}{\lebesgue}$ to the observation fiber $\obsfib$, which is then renormalized w.r.t.~the canonical volume measure on the fiber.
    However, in general, the law of total probability fails to hold for this restricted measure.
    Our main result, \cref{thm:disintegration-density}, shows that an additional corrective factor of $\norm{\nabla h(x)}_2^{-1}$ needs to be multiplied into this restricted density to obtain the density of the true conditional distribution, which is defined implicitly through the notion of a \emph{disintegration}.
    Intuitively speaking, this is due to the fact that, for the law of total probability to hold, the predictive probability mass has to be distributed between fewer fibers where the fibers lie less dense, which is the case if $\norm{\nabla h(x)}_2$ is small (see plots above).
    The plots show a stark disagreement between the restricted and disintegration densities.
    Notably the modes of the measures differ maximally.
    See \cref{ex:gaussian-ellipse} for details.
  }
  \label{fig:mode_example}
\end{figure}

Transferring intuition from the density setting, a ``folklore'' result in machine learning conflates disintegration of a distribution with restriction of the distribution to the submanifold on which the conditional is supported.
For examples of this see e.g.\ \citep[Figure 2.18]{bishop_pattern_2006}, which describes the von Mises distribution as being obtained by ``considering a two-dimensional Gaussian distribution [...] and conditioning on the unit circle''.
In fact this is a \emph{restriction} rather than a conditional, and our novel results highlight that these can differ dramatically.
In \cref{fig:mode_example} we provide a (disturbingly) simple example related to this distribution in which restriction and disintegration result in radically different distributions.
This example takes an ambient standard Gaussian measure on $\R^2$ (\cref{fig:mode_example:prior}) and contrasts restricting (\cref{fig:mode_example:restriction}) and disintegrating (\cref{fig:mode_example:disintegration}) the measure on successively larger ellipses; see \cref{ex:gaussian-ellipse} for more detail.
The colour maps on contours in \cref{fig:mode_example:restriction,fig:mode_example:disintegration} depict appropriate densities for the restriction or disintegration, respectively.
Notably, for smaller ellipses, the mass concentrates at different poles between the two approaches.

A special case of our characterisation result on $\R^d$ is reproduced below:

\begin{theorem}
  \label{thm:disint-density-Rd}
  Let $\prior$ be a Borel probability measure on $\R^d$ that has a density $\dv{\prior}{\lebesgue}$ w.r.t.~the Lebesgue measure $\lebesgue$ on $\R^d$.
  Let $\obsop \colon \R^d \to \R^n$ such that $\prior$-almost all points $\latval \in \R^d$ are $C^1$ regular (see \cref{def:regular-point}).
  Then there exists a version of the disintegration $\priordisint[]$ with probability densities\footnote{The densities are with respect to the restriction $\lebesgue_{\obsfib}$ of the Lebesgue measure to the observation fiber $\obsfib$ (see \cref{def:restricted-measure}).}
  \begin{equation*}
    \dv{\priordisint}{\lebesgue_{\obsfib}}{(\latval)}
    \propto
    \dv{\prior}{\lebesgue}{(\latval)} \cdot \abs{\det(\jacobian{\obsop}(\latval)\rvert_{\kernel{\jacobian{\obsop}(\latval)}^\perp})}\inv
  \end{equation*}
  for all $C^1$ regular values $\obs \in \R^n$, where $\lebesgue_{\obsfib}$ is the restricted Lebesgue measure to the fiber $\obsfib$ (see \cref{def:restricted-measure}).
\end{theorem}

\cref{thm:disint-density-Rd} shows that the disintegration density differs from the restriction density by a corrective factor depending on the Jacobian of $\obsop$.
See \cref{thm:disintegration-density} for further details. In the case where $\obsop$ is linear, this Jacobian is constant, and we hence recover the classical proportionality formula for densities $\dv*{\priordisint}{\lebesgue_{\obsfib}}
    \propto
    \dv*{\prior}{\lebesgue}$.

To demonstrate the importance of our results we focus on the problem of mode computation, i.e., computation of maximum a-posteriori (MAP) points in Bayesian computation.
MAP estimation summarises a distribution with a single, representative estimate that can be interpreted informally as a ``point of maximal probability''.
This is particularly significant in Bayesian inference.
In many practical settings sampling from a full posterior is computationally expensive and a point estimate is sufficient for decision-making or further modelling.
MAP estimates are also of central use in many posterior approximation settings, where an approximation to the posterior is often constructed to have the same mode as the true posterior.

Modes are traditionally defined as maximisers of a Lebesgue density, or more generally, as centers that supremise a particular ratio of open metric balls compared to other candidate centers as the radius of the balls tends to zero.
The aforementioned difficulties in characterising disintegrations have led some works to propose that MAP points be computed by constraining this maximisation problem to the submanifold on which the conditional is supported---in other words, computing modal points of the restricted measure rather than the disintegrated measure---often resulting in a computationally tractable procedure.
We refer to such points as \emph{restricted modes}, and our work highlights that restricted modes can differ drastically from modes of a disintegration.
To give some specific examples of this practice:
\begin{itemize}
    \item \citet[Definition 3]{Chen2024GaussianConditionalMAP} defines a \emph{conditional mode} precisely as above to construct samplers for Gaussian processes conditioned on a particular class of nonlinear, noise-free observations.
    In fact, this computes a \emph{restricted mode}, potentially meaning that the posterior approximation techniques they propose approximate the restriction rather than the conditional or, in the worst case, do not approximate either distribution.
    \item \citet{Cinquin2024FSPLaplace} propose a Laplace approximation for Bayesian neural networks which computes the MAP estimate of the network weights under a function-space prior that is restricted to the set of functions representable by the neural network.
    \item \citet[Section~3]{Tronarp2023BayesianODEMAP} defines the MAP estimator of a probabilistic ODE solver again as above, resulting in similar potential pitfalls to \cite{Chen2024GaussianConditionalMAP}.
\end{itemize}
A special case of our novel result \cref{thm:rcp-weak-modes}, as it applies on $\R^d$, is reproduced below, to highlight that modes of a disintegration meaningfully differ from restricted modes in a way that depends on the observation operator $\obsop$.
\begin{theorem}
  \label{thm:rcp-weak-modes-Rd}
  Let $\prior$ be a Borel probability measure on $\R^d$ that has a continuous density $\dv{\prior}{\lebesgue}$ w.r.t.~the Lebesgue measure $\lebesgue$ on $\R^d$.
  Let $\obsop \colon \R^d \to \R^n$ such that $\prior$-almost all points $\latval \in \R^d$ are $C^1$ regular (see \cref{def:regular-point}).
  Then there is a version of the disintegration $\priordisint[]$ whose weak modes are given by
  \begin{equation*}
    \begin{alignedat}{4}
      \argmin_{{\latval} \in \R^d} & \quad && - \log \dv{\prior}{\lebesgue}{({\latval})} + \log \abs{\det \jacobian{{\obsop}}({\latval}) \vert_{\kernel{\jacobian{{\obsop}}({\latval})}^\perp}} \\
      \textnormal{s.t.} &&& {\obsop}({\latval}) = {\obs}
    \end{alignedat}
  \end{equation*}
  for all $C^1$ regular values $\obs \in \R^n$.
\end{theorem}
The proof can be found in \cref{sec:rcp-weak-modes}.
Restricted modes are obtained by dropping the $\log\det$ term in \cref{thm:rcp-weak-modes-Rd}, and so only coincide with modes of the disintegration in settings where this term is constant, such as for linear $\obsop$.

\cref{fig:mode_example:densities} graphically illustrates the severe impact this can have by contrasting densities on a particular ellipse in the wider setting of \cref{fig:mode_example}.
Disintegration and restriction show almost opposite modes in this setting.

\subsection{Related Work}
\label{sec:related-work}
\paragraph{Disintegrations}
Disintegrations have received limited attention in the literature.
\citet{tjur_constructive_1975} and \citet[Chapter~9]{tjur_probability_1980} each developed disintegrations (also referred to as decompositions) and studied their existence, almost-sure uniqueness and continuity properties under an assumption of a continuous observation operator $\obsop$ and completely regular $\latsp,\obssp$.
\citet{dellacherie_probabilities_1978} provided existence results under weaker assumptions (only requiring $\obsop$ to be measurable, but with a separability assumption on $\obssp$).
\cite{chang_conditioning_1997} advocate for disintegrations as the prototype for rigorous conditional distributions, rather than relying on intuition derived from taking limits of conditionals derived in more ``well-behaved'' settings.

\cite{diaconis_sampling_2013} obtain a related formula to our construction of disintegrations, in the special case of Euclidean spaces. Instead of constructing disintegrations from the ground up, they apply the co-area formula. Other recent works include \cite{possobon_geometric_2022}, which studies the regularity of the disintegration map as a function of the observation.
\cite{cockayne_bayesian_2019} propose a sampling routine which, under some assumptions, approximately samples a disintegration.
However these assumptions are difficult to verify.

\paragraph{Modes}

Recent years have seen a surge in research into the theoretical definition of modal points of distributions, enabling a departure from the familiar notion of a MAP point as the maximiser of a Lebesgue density.
The modern, more general view was initially conceived by \citet{Dashti2013MAPBayesianInverseProblems}, who proposed to seek points that maximise the limit of a ratio of small ball probabilities with different candidate centers, tractably realised through the connection to an \emph{Onsager-Machlup} functional.
They studied the properties of MAP points in separable Hilbert spaces for distributions dominated by a Gaussian reference measure.
The MAP estimator of \citet{Dashti2013MAPBayesianInverseProblems} was termed a \emph{strong MAP estimator} by \citet{Helin2015MaximumPosterioriProbability}, who introduced the idea of a \emph{weak MAP estimator}, weakening the definition by comparing balls centered at a candidate MAP point to shifted balls, where shifts are limited to a topologically dense subset.
\citet{Lie2018EquivalenceWeakStrong} showed that these two definitions are in fact equivalent under an additional uniformity condition on the measure.

In more recent works, \citet{clason_generalized_2019} introduced the idea of a \emph{generalised mode} that studies ball ratios in a more general asymptotic sense, studying the limiting ball ratio along sequences of points approaching candidate MAP points.
This prompted \citet{klebanov_classification_2025} to note the similarity of the three existing mode definitions and generate and study a family of feasible alternative definitions by constructing a grammar from mathematical symbols in the definitions.
Of the 282 ``grammatically correct'' definitions generated, ten were deemed meaningful in terms of satisfying some natural properties to expect of such definitions.
\citet{Ayanbayev2022GammaConvergence} studies $\Gamma$-convergence of global weak modes and demonstrates some of the difficulties extending this approach to strong modes.
\citet{Lambley2023StrongMAP} studies strong modes of Bayesian posterior measures in Banach spaces, under Gaussian priors, and demonstrates their equivalence with other notions of modes.

Specifically studying MAP points in the setting of this paper (i.e.\ when the support of the conditional measure has no prior mass) has received much less attention.
As discussed above both \citet{Chen2024GaussianConditionalMAP} and \citet{Tronarp2023BayesianODEMAP} construct what we call \emph{restricted modes} in pursuit of the MAP of a disintegration.
\citet{chen_solving_2021} and \citet{Cinquin2024FSPLaplace} adopt a similar approach.

\subsection{Contributions}

The main contributions of our work are as follows:
\begin{itemize}
    \item We prove several intermediate important properties of disintegrations, leading to a formula for the density of a disintegration on a Riemannian manifold $\latsp$, with-respect-to the volume measure on the fibers $\obsop^{-1}(\set{y})$.
    \item We characterise modes of disintegrations and of (Riemannian) restricted measures, highlighting their disagreement in general.
    \item We illustrate numerically the potentially large extent of disagreement between the two notions of modes on a simple toy problem, that is nevertheless representative of contemporary applications.
\end{itemize}

\subsection{Structure of the Paper}

The rest of the paper proceeds as follows.
In \cref{sec:disintegration} we introduce disintegrations and prove several intermediate results leading to our constructive definition of disintegrations on Riemannian manifolds in \cref{sec:disintegration-manifolds}.
\cref{sec:rcp-weak-modes} focuses on how these ideas pertain to the computation of modes of disintegrations, with \cref{sec:rcp-conditional-modes-vs-weak-modes} proving disagreement of restricted modes and disintegration modes and \cref{sec:failure-modes} providing a more explicit characterisation of failure modes in $\R^d$.
We conclude with some discussion on the implications of these results in \cref{sec:discussion}.

\section{Constructive Disintegration}
\label{sec:disintegration}
We begin by introducing transition and Markov kernels, which are fundamental to the definition of a disintegration.
Let $(\latsp, \sigalg_\latsp)$ and $(\obssp, \sigalg_\obssp)$ be measurable spaces.
\begin{definition}[Transition Kernel]
  \label{def:transition-kernel}
  A function $\kappa\condps{\openarg \given \openarg} \colon \sigalg_\latsp \times \obssp \to [0, \infty]$ is called a \emph{transition kernel} from $\obssp$ to $\latsp$ if
  \begin{enumerate}[label=(\roman*)]
    \item \label{item:transition-kernel-measure}
      $\kappa\condps{\openarg \given \obs}$ is a measure on $(\latsp, \sigalg_\latsp)$ for all $\obs \in \obssp$, and
    \item \label{item:transition-kernel-measurable}
      $\kappa\condps{\latset \given \openarg}$ is $\sigalg_\obssp$-$\borelsigalg{[0, \infty]}$-measurable for all $\latset \in \sigalg_\latsp$.
  \end{enumerate}
  If $\kappa\condps{\openarg \given \obs}$ is a probability measure for all $\obs \in \obssp$, then $\kappa$ is called a \emph{Markov kernel}.
\end{definition}
We extend measure-theoretic operations like pushforwards $\pushfw{\kappa}{f}$, integrals $\kappa\condps{f, \obs}$ and $f \kappa$, as well as Radon-Nikodym derivatives $\dv{\kappa}{\nu}$ to transition kernels.

Next we introduce disintegrations as Markov kernels satisfying certain regularity properties.
To this end, let $\prior$ be a measure on $(\latsp, \sigalg_\latsp)$ and
$\obsop \colon \latsp \to \obssp$ a $\sigalg_\latsp$-$\sigalg_\obssp$-measurable function.
\begin{definition}[Disintegration]
  \label{def:disintegration}
  A \emph{disintegration} of $\prior$ with respect to $\obsop$ is a transition kernel $\priordisint[\openarg](\openarg)$ from $\obssp$ to $\latsp$ such that
  \begin{enumerate}[label=(\roman*)]
    \item \label{item:disintegration-support}
      $\priordisint(\latsp \setminus \obsfib) = 0$,
    \item \label{item:disintegration-law-of-total-measure}
      $\prior(\latset) = \int_{\obssp} \priordisint(\latset) \predictive(\dd{\obs})$ for any $\latset \in \sigalg_\latsp$.
  \end{enumerate}
\end{definition}

If $\prior$ is a probability measure, then the elements of the disintegration are sometimes referred to as \emph{regular conditional probability measures}.
In this case, \labelcref{item:disintegration-law-of-total-measure} is equivalent to the law of total probability when conditioning on events of nonzero probability.

\begin{remark}
    By the monotone convergence theorem, \cref{def:disintegration}\ref{item:disintegration-law-of-total-measure} is equivalent to $\prior(f) = \int_{\obssp} \priordisint(f) \predictive(\dd{\obs})$ for any non-negative measurable $f$.
\end{remark}

For the remainder of \cref{sec:disintegration}, we will always work under the following set of assumptions:
\begin{assumption}\label{asm:disintegration-conditions}
  \leavevmode
    \begin{enumerate}[label=(\roman*)]
        \item \label{item:latsp} $\latsp$ is a metrizable topological space equipped with its Borel $\sigma$-algebra, i.e.~$\sigalg_\latsp = \mathcal B(\latsp)$.

        \item \label{item:obssp} $\sigalg_\obssp$ is countably generated and contains all singleton sets $\{y\}$.

        \item \label{item:measure} $\prior$ is a $\sigma$-finite Radon measure on $\latsp$ and its pushforward $\predictive$ is also $\sigma$-finite.
    \end{enumerate}
\end{assumption}

\begin{remark}
    The assumption that $\predictive$ is a $\sigma$-finite measure is non-trivial; for example the projection of the Lebesgue measure on $\R^2$ onto a coordinate axis is not $\sigma$-finite.
\end{remark}

Under \cref{asm:disintegration-conditions}, one can prove existence and uniqueness (up to $\predictive$ null sets) of the disintegration.
As discussed in \cref{sec:introduction}, various forms of this result exist under different and potentially more general assumptions.
The version we introduce is from \citet{chang_conditioning_1997}.
\begin{theorem}[Disintegration Theorem; see {\citealp[][Theorem 1]{chang_conditioning_1997}}]
  \label{thm:disintegration}
  There exists a $\predictive$-almost everywhere uniquely defined disintegration $\priordisint[]$, i.e., the function $\obs \mapsto \priordisint(\openarg)$ is $\predictive$-almost everywhere uniquely determined.
\end{theorem}

\begin{remark}
    \cite{chang_conditioning_1997} allow for disintegrating $\prior$ with respect to a different measure than $\predictive$. For simplicity and clarity, we focus in this paper exclusively on the standard case of disintegrating $\prior$ with respect to its pushforward, though our results can be extended to the more general case.
\end{remark}

When thinking of conditioning, most practitioners will first think of Bayes' rule. To argue that disintegrations provide a rigorous way of conditioning measures more generally, it is crucial to show that, under some assumptions, Bayes' theorem can be recovered from \cref{def:disintegration}.
This is the focus of our first result; a statement of Bayes' theorem purely in terms of disintegrations.

In the next theorem we think of $\latsp = \sublatsp\times \subobssp$, where $\sublatsp$ should be thought as the parameter space, $\subobssp$ the data space, and $\prior$ the joint (probability) measure over  $\sublatsp\times \subobssp$. Hence $\pi_{\sublatsp\star}\prior$ is the prior, $\pi_{\sublatsp\star}(\disint{\prior}{\pi_{\subobssp}}[\subobs])$ is the posterior given data $\subobs\in\subobssp$, and $\pi_{\subobssp \star}(\disint{\prior}{\pi_{\sublatsp}}[\sublatval])$ is the likelihood at parameter $\sublatval\in \sublatsp$.

The above exposition makes it clear that we must disintegrate $\prior$ with respect to both the coordinate projections $\pi_{\subobssp}$ and $\pi_{\sublatsp}$.
Hence, we require $\sublatsp$ and $\subobssp$ to each serve the role of both $\latsp$ and $\obssp$ in \cref{def:disintegration}, and hence to satisfy the assumptions on both spaces from (\cref{asm:disintegration-conditions}).
We phrase these succinctly in the following theorem:
\begin{theorem}[Bayes' Theorem]
    Suppose that both $\sublatsp$ and $\subobssp$ are metrizable separable topological spaces equipped with their Borel $\sigma$-algebra, that $\latsp = \sublatsp \times \subobssp$, and that $\prior$ is a finite Radon measure on $\latsp$. Further suppose that there is a $\sigma$-finite measure $\nu$ on $\subobssp$ such that $\pi_{\subobssp\star}(\disint{\prior}{\pi_{\sublatsp}}[\sublatval]) \ll \nu$ for $\pi_{\sublatsp\star}\mu$-almost all $\sublatval$. Then
    \begin{equation*}
        \dv{(\pi_{\sublatsp\star}(\disint{\prior}{\pi_{\subobssp}}[
          \subobs]))}{(\pi_{\sublatsp\star}\prior)}{(\sublatval)} = \frac{\dv*{(\pi_{\subobssp \star}(\disint{\prior}{\pi_{\sublatsp}}[\sublatval]))}{\nu}{(\subobs)}}{\dv*{(\pi_{\subobssp\star}\prior)}{\nu}{(\subobs)}}
    \end{equation*}
    where the right-hand-side is well-defined for $\pi_{\subobssp\star}\prior$-almost every $\subobs$.
\end{theorem}
\begin{proof}
    We will show
    \begin{equation*}
        \disint{\prior}{\pi_{\subobssp}}[\subobs](\dd \subobs'\times \dd \sublatval') = \frac{\dv*{(\pi_{\subobssp \star}(\disint{\prior}{\pi_{\sublatsp}}[\sublatval']))}{\nu}{(\subobs)}}{\dv*{(\pi_{\subobssp\star}\prior)}{\nu}{(\subobs)}}\pi_{\sublatsp\star}\prior(\dd \sublatval') \delta_{\subobs}(\dd \subobs')
    \end{equation*}
    where $\delta_{\subobs}$ is the Dirac measure at $\subobs$. Clearly, the right-hand-side is supported on $\pi^{-1}_{\subobssp}(\{\subobs\}) = \{\subobs\} \times\sublatsp$, so it satisfies \cref{def:disintegration}\labelcref{item:disintegration-support}. Now for $\subobsset\subset \subobssp$ and $\sublatset\subset \sublatsp$ measurable, define
    \begin{equation*}
      \tilde{\subobsset}
      \defeq \subobsset \cap \set{\subobs \in \subobssp \where \dv*{(\pi_{\subobssp\star}\prior)}{\nu}{(\subobs)} \neq 0}.
    \end{equation*}
    Now $\nu$ is absolutely continuous with respect to $\pi_{\subobssp\star}\prior$ when restricted to $\tilde{Z}$ with
    \begin{equation*}
      \dv{\nu}{(\pi_{\subobssp\star}\prior)}{(\subobs)}
      = \left( \dv{(\pi_{\subobssp\star}\prior)}{\nu}{(\subobs)} \right)^{-1}.
    \end{equation*}
    Then
    \begin{align*}
      & \int_{\subobssp} \left(\int_\sublatset\frac{\dv*{(\pi_{\subobssp \star}(\disint{\prior}{\pi_{\sublatsp}}[\sublatval]))}{\nu}{(\subobs)}}{\dv*{(\pi_{\subobssp\star}\prior)}{\nu}{(\subobs)}}\pi_{\sublatsp\star}\prior(\dd \sublatval)\right)\cdot \delta_{\subobs}(\subobsset) \pi_{\subobssp\star}\prior(\dd \subobs) \\
      & = \int_{\tilde{\subobsset}}\int_\sublatset \frac{\dv*{(\pi_{\subobssp \star}(\disint{\prior}{\pi_{\sublatsp}}[\sublatval]))}{\nu}{(\subobs)}}{\dv*{(\pi_{\subobssp\star}\prior)}{\nu}{(\subobs)}} \pi_{\sublatsp\star}\prior(\dd \sublatval)\pi_{\subobssp\star}\prior(\dd\subobs) \\
      & =\int_\sublatset\int_{\tilde{\subobsset}} \frac{\dv*{(\pi_{\subobssp \star}(\disint{\prior}{\pi_{\sublatsp}}[\sublatval]))}{\nu}{(\subobs)}}{\dv*{(\pi_{\subobssp\star}\prior)}{\nu}{(\subobs)}} \pi_{\subobssp\star}\prior(\dd\subobs)\pi_{\sublatsp\star}\prior(\dd \sublatval) \\
      & =\int_\sublatset\int_{\tilde{\subobsset}} \dv{(\pi_{\subobssp \star}(\disint{\prior}{\pi_{\sublatsp}}[\sublatval]))}{\nu}{(\subobs)} \dv{\nu}{(\pi_{\subobssp\star}\prior)}{(\subobs)} \pi_{\subobssp\star}\prior(\dd\subobs)\pi_{\sublatsp\star}\prior(\dd \sublatval) \\
      & =\int_\sublatset\int_{\tilde{\subobsset}} \dv{(\pi_{\subobssp \star}(\disint{\prior}{\pi_{\sublatsp}}[\sublatval]))}{\nu}{(\subobs)}\nu(\dd \subobs)\pi_{\sublatsp\star}\prior(\dd \sublatval) \\
      & = \int_\sublatset \pi_{\subobssp \star}(\disint{\prior}{\pi_{\sublatsp}}[\sublatval])(\tilde{\subobsset}) \pi_{\sublatsp\star}\prior(\dd \sublatval) \\
      & = \int_\sublatset \disint{\prior}{\pi_{\sublatsp}}[\sublatval]( \sublatsp \times \tilde{\subobsset}) \pi_{\sublatsp\star}\prior(\dd \sublatval) \\
      & = \int_{\sublatsp} \disint{\prior}{\pi_{\sublatsp}}[\sublatval]( \sublatset \times \tilde{\subobsset}) \pi_{\sublatsp\star}\prior(\dd \sublatval) \\
      & = \prior(\sublatset \times \tilde{\subobsset}) \\
      & = \prior(\sublatset \times \subobsset),
    \end{align*}
    where the first and last equality hold because $\pushfw{\prior}{{\pi_{\subobssp}}}(\subobsset \setminus \tilde{\subobsset}) = 0$, and second equality holds by Fubini's theorem.

    The sets of the form $\sublatset \times \subobsset \in \mathcal A_\latsp$ are stable under finite intersections and generate $\mathcal A_\latsp = \mathcal A_{\sublatsp} \otimes \mathcal A_{\subobssp}$, so we have shown that our expression satisfies \cref{def:disintegration}\labelcref{item:disintegration-law-of-total-measure}, and hence is indeed the disintegration $\disint{\prior}{\pi_{\subobssp}}[\subobs]$. The theorem then follows by projecting onto $\sublatsp$.
\end{proof}

\subsection{Disintegration Building Blocks}
Because of the almost everywhere nature of \cref{thm:disintegration}, all equalities in this section are meant $\predictive$-almost everywhere.

In this subsection we establish six scaffolding results about disintegrations that are of independent interest. 
We will later demonstrate in \cref{sec:disintegration-manifolds} how they can be combined to give, under certain regularity assumptions, an explicit formula for disintegrations on manifolds in \cref{thm:disintegration-density}.
We anticipate that other results further characterising disintegrations can be obtained from these lemmas.

The results are, we believe, quite intuitive and straightforward to explain narratively, so we summarise them briefly here.
\Cref{lem:disintegration-bijection} provides the most elementary building block, showing that disintegrations under the identity map yield Dirac measures. \Cref{lem:disintegration-product} gives an expression for the disintegration of a product measure with respect to a coordinate map in terms of the product of the disintegration. \Cref{lem:disintegration-pushforward} shows that the disintegration of the pushforward is the pushforward of the disintegration. \Cref{lem:disintegration-dominated} shows that the disintegration of a measure dominated by some base measure is dominated by the disintegration of the base measure, and its density is a renormalization of the prior density. \Cref{lem:disintegration-restriction} shows that the disintegration of a measure restricted to some set is the renormalization of the restriction of the disintegration of the measure. Finally, \cref{prop:disintegration-equiv-obs} shows that disintegrations are invariant to composition of the observation map with a bijection.
\begin{lemma}[Disintegration w.r.t.~the Identity]
    \label{lem:disintegration-bijection}
    If $\obssp = \latsp$ and $\obsop\colon \latsp \to \obssp$ is the identity map, then
    \begin{equation*}
        \priordisint[\obs] = \delta_{\obs}
    \end{equation*}
    the Dirac measure at $\obs$.
\end{lemma}
\begin{proof}
    $\priordisint[\obs]$ is supported on $\obsop\inv(\{\obs\}) = \{\obs\}$. Since $\priordisint[\obs]$ is a probability measure, we necessarily have $\priordisint[\obs] = \delta_{y}$. We check: for $\latset\subset \latsp$ measurable,
    \begin{equation*}
        \int_{\obssp} \delta_{y}(\latset) \predictive(\dd{\obs})= \int_{\latsp} \mathbbm{1}_{\latset}{(x)}\prior(\dd{x}) = \int_{\latset}\prior(\dd{x})= \prior(\latset)
    \end{equation*}
    as required.
\end{proof}

In the next lemma $\latsp_1, \latsp_2$ are under the same assumptions as $\latsp$ (\cref{asm:disintegration-conditions}\ref{item:latsp}), and $\mu_1,\mu_2$ are under the same assumptions as $\mu$ (\cref{asm:disintegration-conditions}\ref{item:measure}) but on $\latsp_1, \latsp_2$ respectively.
\begin{lemma}[Disintegration of Product Measure]
  \label{lem:disintegration-product}
  Suppose $\latsp = \latsp_1 \times \latsp_2$, $\prior = \prior_1 \times \prior_2$, and that $\mu_2$ is finite.
  Let $\obsop_1 : \latsp_1 \to \obssp$, and $\obsop = \obsop_1 \circ \pi_1 \colon \latsp \to \obssp, (\latval_1, \latval_2) \mapsto \obsop_1(\latval_1)$.
  Then, for $\obs \in \obssp$,
  \begin{equation*}
    \prior_2(\latsp_2)\priordisint
    = \disint{\prior_1}{\obsop_1}[\obs] \times \prior_2.
  \end{equation*}
\end{lemma}
\begin{proof}
  Property \labelcref{item:disintegration-support} in \cref{def:disintegration} for $\priordisint[]$ follows from the corresponding property of $\disint{\prior_1}{\obsop_1}$:
  \begin{equation*}
      (\disint{\prior_1}{\obsop_1}[\obs] \times \prior_2)(\latsp \setminus \obsfib) = \disint{\prior_1}{\obsop_1}[\obs](\latsp_1 \setminus h_1\inv(\{y\}))\cdot \mu_2(\latsp_2) = 0.
  \end{equation*}
  Now for $\latset_1\subset\latsp_1$, $\latset_2\subset\latsp_2$ measurable
  \begin{align*}
    \int_{\obssp} (\disint{\prior_1}{\obsop_1}[\obs] \times \prior_2)(\latset_1\times\latset_2)\predictive(\dd{\obs}) &= \int_{\latsp_1} \disint{\prior_1}{\obsop_1}[\obs](\latset_1) \cdot \prior_2(\latset_2)\predictive(\dd{\obs}) \\
    &= \mu_2(\latset_2)\int_{\latsp_1}\disint{\prior_1}{\obsop_1}[\obs](\latset_1)\predictive(\dd{\obs}) \\
    &= \mu_2(\latsp_2)\mu_2(\latset_2)\int_{\latsp_1}\disint{\prior_1}{\obsop_1}[\obs](\latset_1)\pushfw{\prior_1}{{\obsop_1}}(\dd{\obs}) \\
    &= \mu_2(\latsp_2)\mu_1(\latset_1)\mu_2(\latset_2) \\
    &= \mu_2(\latsp_2)\mu(\latset_1\times\latset_2).
  \end{align*}
  where we used $\predictive = \prior_2(\latsp_2) \pushfw{\prior_1}{{\obsop_1}}$. The sets of the form $X_1\times X_2\in \mathcal A_{\latsp}$ are stable under finite intersections and generate $\mathcal A_{\latsp} = \mathcal A_{\latsp_1}\otimes \mathcal A_{\latsp_2}$, so this is enough to conclude the proof.
\end{proof}
\begin{remark}
    It is important that $\prior_2$ is a finite measure in order for $\pi_{1\star}\prior$ and hence $\predictive$ to be $\sigma$-finite.
\end{remark}

In the following lemma $\tilde\latsp$ is under the same assumptions as $\latsp$ (\cref{asm:disintegration-conditions}\ref{item:latsp}).
\begin{lemma}[Disintegration of Pushforward Measure]
  \label{lem:disintegration-pushforward}
  Let $f \colon \latsp \to \tilde{\latsp}$, $g \colon \tilde{\latsp} \to \obssp$ be measurable maps such that $\obsop = g \circ f$.
  Then, for $\obs \in \obssp$,
  \begin{equation*}
    \disint{(\pushfw{\prior}{f})}{g}[y] = \pushfw{(\priordisint[y])}{f}.
  \end{equation*}
\end{lemma}
\begin{proof}
  First checking \labelcref{item:disintegration-support} in \cref{def:disintegration}, we have
  \begin{align*}
    \pushfw{(\priordisint)}{f}(\tilde \latsp \setminus \preim{g}(\set{\obs}))
    & = \priordisint(\preim{f}\pqty{\tilde \latsp \setminus \preim{g}(\set{\obs})}) \\
    & = \priordisint(\latsp \setminus \pqty{\preim{f} \circ \preim{g}} (\set{\obs})) \\
    & = \priordisint(\latsp \setminus \obsfib) \\
    & = 0.
  \end{align*}
  Moreover, for $\latset \subset \latsp$ measurable,
  \begin{equation*}
      \int_{\obssp} \pushfw{(\priordisint)}{f}(\latset) \pushfw{(\pushfw{\prior}{f})}{g}(\dd{\obs})
      = \int_{\obssp} \priordisint(\preim{f}(\latset)) \pushfw{\prior}{\obsop}(\dd{\obs})
      = \prior(\preim{f}(\latset))
      = \pushfw{\prior}{f}(\latset).
  \end{equation*}

\end{proof}
In the next lemma $\nu$ is under the same assumptions as $\prior$ (\cref{asm:disintegration-conditions}\ref{item:measure}).
\begin{lemma}[Disintegration of Dominated Measure]
  \label{lem:disintegration-dominated}
  Suppose $\prior \ll \nu$.
  Then for $\obs \in \obssp$, $\priordisint[\obs]$ is absolutely continuous with respect to $\disint{\nu}{\obsop}[\obs]$ and
  \begin{equation*}
    \dv{\priordisint[\obs]}{\disint{\nu}{\obsop}[\obs]}
    = \frac{\dv*{\prior}{\nu}}{\disint{\nu}{\obsop}[\obs](\dv*{\prior}{\nu})}
  \end{equation*}
  where $\disint{\nu}{\obsop}[\obs](\dv*{\prior}{\nu})>0$ $\predictive$-almost everywhere.
\end{lemma}
\begin{proof}
  Property \labelcref{item:disintegration-support} in \cref{def:disintegration} for $\priordisint[]$ follows from the corresponding property of $\disint{\nu}{\obsop}$.
  It remains to show property \labelcref{item:disintegration-law-of-total-measure}.
  Note that for $\obsset \subset \obssp$ measurable
  \begin{align*}
    \int_\obsset \disint{\nu}{\obsop}[\obs](\dv{\prior}{\nu}) \pushfw{\nu}{\obsop}(\dd{\obs})
    & = \int_{\obssp} \mathbbm{1}_{\obsset}(\obs) \cdot \disint{\nu}{\obsop}[\obs](\dv{\prior}{\nu}) \pushfw{\nu}{\obsop}(\dd{\obs}) \\
    & = \int_{\obssp} \disint{\nu}{\obsop}[\obs](\mathbbm{1}_{\obsop\inv(\obsset)} \cdot \dv{\prior}{\nu}) \pushfw{\nu}{\obsop}(\dd{\obs})\\
    & = \nu \pqty{\mathbbm{1}_{\obsop\inv(\obsset)} \cdot \dv{\prior}{\nu}} \\
    & = \prior(\preim{\obsop}(\obsset)) \\
    & = \pushfw{\prior}{\obsop}(\obsset),
  \end{align*}
  so
  \begin{equation*}
    \dv{(\pushfw{\prior}{\obsop})}{(\pushfw{\nu}{\obsop})}{(\obs)}
    = \disint{\nu}{\obsop}[\obs](\dv{\prior}{\nu}).
  \end{equation*}
  Now letting $Y:=\{y\in\obssp: \disint{\nu}{\obsop}[\obs](\dv*{\prior}{\nu}) = 0\}$, we have
  \begin{equation*}
      \predictive(Y) = \int_Y \disint{\nu}{\obsop}[\obs](\dv{\prior}{\nu}) \pushfw{\nu}{\obsop}(\dd{\obs}) = 0
  \end{equation*}
  so $\disint{\nu}{\obsop}[\obs](\dv*{\prior}{\nu})>0$ $\predictive$-almost everywhere.

  Thus, for $\latset \subset \latsp$ measurable,
  \begin{align*}
    \int_{\obssp} \disint{\nu}{\obsop}[\obs](\dv{\prior}{\nu}) \priordisint(\latset) \pushfw{\nu}{\obsop}(\dd{\obs})
    & = \int_{\obssp} \priordisint(\latset) \pushfw{\prior}{\obsop}(\dd{\obs}) \\
    & = \prior(\latset) \\
    & = \nu \pqty{\mathbbm{1}_\latset \cdot \dv{\prior}{\nu}} \\
    & = \int_{\obssp} \disint{\nu}{\obsop}[\obs](\mathbbm{1}_\latset \cdot \dv{\prior}{\nu}) \pushfw{\nu}{\obsop}(\dd{\obs}) \\
    & = \int_{\obssp} \int_\latset \dv{\prior}{\nu}{(\latval)} \disint{\nu}{\obsop}[\obs](\dd{\latval}) \pushfw{\nu}{\obsop}(\dd{\obs})
  \end{align*}
  and hence
  \begin{equation*}
      \disint{\nu}{\obsop}[\obs](\dv{\prior}{\nu})\priordisint[\obs] = \dv{\prior}{\nu}\disint{\nu}{\obsop}[\obs]
  \end{equation*}
  is the unique disintegration of $\prior$ with respect to $\obsop_\star\nu$ (see \citet[Theorem 1]{chang_conditioning_1997}).
\end{proof}

\begin{lemma}[Disintegration of Restricted Measure]
  \label{lem:disintegration-restriction}
  Let $\latset \subset \latsp$ be measurable.
  Then, for $\obs \in \obssp$,
  \begin{equation*}
    \left. \priordisint \right\rvert_\latset
    = \priordisint(\latset) \disint{(\left. \prior \right\rvert_\latset)}{\obsop}[\obs].
  \end{equation*}
\end{lemma}
\begin{proof}
  Let $\nu$ be the measure on $\latsp$ defined by $\nu := \prior(\openarg \cap \latset)$.
  Then $\nu \ll \prior$ with $\dv{\nu}{\prior} = \mathbbm{1}_{\latset}$.
  Thus, applying \cref{lem:disintegration-dominated}, we obtain for any measurable $\latset' \subset \latsp$
  \begin{equation*}
    \priordisint(\latset) \disint{\nu}{\obsop}[\obs](\latset')
    = \priordisint(\mathbbm{1}_{\latset}) \disint{\nu}{\obsop}[\obs](\latset')
    = \int_{\latset'} \mathbbm{1}_{\latset}(\latval) \priordisint(\dd{\latval})
    = \priordisint(\latset' \cap \latset).
  \end{equation*}
  The result then follows by restricting the measures to $\latset$ and noting that since $\nu|_X = \prior|_X$ we have $\left. \disint{\nu}{\obsop}[\obs] \right\rvert_\latset = \disint{(\left. \nu \right\rvert_\latset)}{\obsop}[\obs] = \disint{(\left. \prior \right\rvert_\latset)}{\obsop}[\obs]$.
\end{proof}

We have now established all the lemmas required for our main result (\cref{thm:disintegration-density}) in the following subsection. We will nevertheless prove an additional result that will be of interest in the upcoming discussion. Here $\tilde \obssp$ is under the same assumptions as $\obssp$ (\cref{asm:disintegration-conditions}\ref{item:obssp}).

\begin{proposition}[Disintegration w.r.t.~Equivalent Observations]\label{prop:disintegration-equiv-obs}
    Let $f\colon \obssp \to \tilde\obssp$ be an isomorphism of measure spaces. Then
    \begin{equation*}
        \priordisint = \disint{\prior}{f\circ\obsop}[f(\obs)].
    \end{equation*}
\end{proposition}
\begin{proof}
    Property \labelcref{item:disintegration-support} in \cref{def:disintegration} follows from the fact that 
    \[
    \obsfib = (f\circ\obsop)\inv(\{f(y)\}).
    \] 
    Now for $\latset\subset\latsp$ measurable, by (measure-theoretic) change of variables,
    \begin{equation*}
        \int_{\obssp} \disint{\prior}{f\circ\obsop}[f(\obs)](\latset) \predictive(\dd{\obs}) = \int_{\tilde\obssp} \disint{\prior}{f\circ\obsop}[\tilde\obs](\latset) (f\circ\obsop)_\star(\dd{\tilde \obs}) = \prior(X).
    \end{equation*}
\end{proof}

\subsection{Construction of Disintegrations on Manifolds}
\label{sec:disintegration-manifolds}
We begin by stating some additional assumptions on manifold structure of $\latsp$ that are required for our main result in \cref{sec:disintegration}, \cref{thm:disintegration-density}.
For readers unfamiliar with the differential geometry used in this section, \cref{app:differential_geometry} provides a concise summary of the relevant definitions and results.

\begin{assumption}
  \label{asm:prior-observation}
  \leavevmode
  \begin{enumerate}[label=(\roman*)]
    \item \label{item:X-prior}
      Let $\latsp$ be a $d$-dimensional $C^k$ manifold.
      Assume that $\prior$ is absolutely continuous with respect to a volume measure on $\latsp$.
    \item \label{item:Y-observation}
      Let $\obssp = \R^n$ for some $n \in \N$, and $\obsop \colon \latsp \to \obssp$
      such that the points in $\latsp$ are $\prior$-almost everywhere $C^k$ regular w.r.t.~$\obsop$.
  \end{enumerate}
\end{assumption}

A $C^k$ regular point w.r.t.~$h$ is a point $x\in\latsp$ around which $h$ is $C^k$ and the derivative of $h$ is surjective (\cref{def:regular-point}). By the preimage theorem (\cref{thm:implicit-function}), \Cref{asm:prior-observation} ensures in particular that $\obsfib$ is a $(d-n)$-dimensional $C^k$ submanifold of $\latsp$ for $\predictive$-almost every $y$. We will use this to construct the disintegration $\predictive$-almost everywhere, on these manifolds.

There is no canonical way to restrict an ambient measure or volume form from $\latsp$ to the generally lower dimensional fiber $\obsfib$ directly\footnote{Even when the pushforward of the volume measure $h_\star\volmeas{\latsp}$ is $\sigma$-finite and we can disintegrate $\volmeas{\latsp}$, such disintegration cannot be interpreted as a canonical \emph{restriction} to the fiber, since we will see the disintegration depends on $\obsop$ through more than just the fiber $\obsfib$ (\cref{rmk:disintegration-fibers}).}.
However, if we we equip $\latsp$ with a Riemannian metric then we obtain a Riemannian volume measure $\volmeas{\latsp}$, and there \emph{is} a canonical restriction of this metric.
The restricted Riemannian metric on $\obsfib$ then induces a Riemannian volume measure $\volmeas{\obsfib}$ on $\obsfib$. 
This construction allows us more generally to restrict the measure $\mu$, as we will see in the next definition.

\begin{remark}\label{rmk:volume-measures}
    In \cref{asm:prior-observation}\ref{item:X-prior}, by a volume measure we mean a measure induced by a $d$-form (\cref{def:volume-measure}). All such volume measures are absolutely continuous with respect to any Riemannian volume measure $\volmeas{\latsp}$ (\cref{def:riemann-volume}), thus we could equivalently assume in \cref{asm:prior-observation}\ref{item:X-prior} that $\mu$ is absolutely continuous with respect to $\volmeas{\latsp}$.
\end{remark}

\begin{definition}[Riemannian Restricted Measure]
  \label{def:restricted-measure}
    Suppose $\latsp$ is a $C^k$ Riemannian manifold, $\tilde {\mathbb X}$ a $C^k$ ($k\geq 1$) Riemannian submanifold of $\latsp$ and $\mu$ is absolutely continuous with respect to $\volmeas{\latsp}$. Then the \textit{Riemannian restricted measure} $\mu_{\tilde{\mathbb X}}$ is the measure given by $\mu_{\tilde {\mathbb X}} := \dv{\mu}{\volmeas{\latsp}}\volmeas{\tilde {\mathbb X}}$.
\end{definition}
\begin{remark} 
    Note the different notation used for a measure that has been restricted in a \emph{measure-theoretic} sense, as in \cref{lem:disintegration-restriction}, i.e.\ $\left. \prior \right\rvert_{\tilde \latsp}$, compared to the \emph{Riemannian} restricted measure from \cref{def:restricted-measure}, i.e.\ $\mu_{\tilde{\mathbb X}}$.
    The two are generally distinct. 
    In particular, whenever $\dim \tilde \latsp< \dim \latsp$, if $\prior$ is absolutely continuous with respect to a volume measure on $\latsp$ then $\left. \prior \right\rvert_{\tilde \latsp}$ is a zero measure, so not a useful notion of restriction for the purposes of subsequent results.
\end{remark}
For $\obs\in\obssp$, $\priordisint$ is a probability measure supported on $\obsfib$. This could lead one to ask whether disintegrations $\priordisint[\obs]$ correspond to renormalizations of Riemannian restricted measures $\mu_{\obsfib}$, under \cref{asm:prior-observation}. Perhaps surprisingly they do not. Riemannian restricted measures depend on the choice of Riemannian metric while, by definition, disintegrations do not.

Next we state and prove \Cref{thm:disintegration-density}, our main result for this section.
The theorem gives an explicit expression for disintegrations under \cref{asm:prior-observation}, showing that they differ from renormalized Riemannian restricted measures by a determinant-like term.
This term quantifies how dense fibers lie locally along directions orthogonal to the fiber.
Intuitively speaking, regions in which the fibers lie less dense need to be equipped with higher disintegration density for the law of total probability (\cref{def:disintegration}\ref{item:disintegration-law-of-total-measure}) to hold.
For more concrete, visual intuition, see \cref{fig:mode_example} and \cref{ex:gaussian-ellipse}.

\begin{theorem}
  \label{thm:disintegration-density}
  Under \cref{asm:prior-observation} with $k=1$, equip $\latsp$ with a Riemannian metric and denote by $\volmeas{\latsp}$ the corresponding volume measure on $\latsp$.
  Then $\priordisint$ is absolutely continuous with respect to the Riemannian volume measure $\volmeas{\obsfib}$ on the Riemannian submanifold $\obsfib$ and
  \begin{equation*}
      \dv{\priordisint}{\volmeas{\obsfib}}{(\latval)} \propto
    \dv{\prior}{\volmeas{\latsp}}{(\latval)} \cdot \abs{\volmeas{\kernel{\jacobian{\obsop}(x)}^\perp}(x)[\nabla\obsop(x)]}\inv.
  \end{equation*}
\end{theorem}
\begin{proof}
  Let $\obs\in \R^n$ a $C^1$ regular value of $h$ and $\latval \in \obsfib$. By the implicit function theorem (\cref{thm:implicit-function}) there are bounded neighbourhoods $U$ of $x$ in $\latsp$, $V$ of $y$ in $\obssp=\R^n$, $W$ of $ 0$ in $\R^{d-n}$ and a $C^1$ diffeomorphism $ \varphi\colon U \to V\times W$ such that $\pi_V \circ \varphi = \obsop$, where $\pi_V\colon V\times W \to V$ is the projection. Let $\lebesgue_{V\times W}$, $\lebesgue_V$, $\lebesgue_W$ be the Lebesgue measures on $V\times W$, $V$ and $W$ respectively. Then, by \cref{eq:riemann-volume-measure-pushforward}, the pushforward of the Riemannian volume measure on $U$ through the chart is given by\footnote{We abstract away the manifold variable in the notation of this proof, e.g.~$\volmeas{U}[\nabla  \varphi\circ\varphi\inv]$ stands for the function $z\mapsto \volmeas{U}(z)[\nabla  \varphi(\varphi\inv(z))]$}
  \begin{equation*}
       \varphi_\star\volmeas{U} = \abs{\volmeas{U}[\nabla  \varphi\circ\varphi\inv]}\inv \lebesgue_{V\times W} = \abs{\volmeas{U}[\nabla  \varphi\circ\varphi\inv]}\inv(\lebesgue_V\times \lebesgue_W).
  \end{equation*}
  So, disintegrating this measure
  \begin{align*}
      \disint{( \varphi_\star\volmeas{U})}{\pi_V}[\obs] &\propto \abs{\volmeas{U}[\nabla  \varphi\circ\varphi\inv]}\inv\disint{\lebesgue_{V\times W}}{\pi_V}[\obs] &\text{(\cref{lem:disintegration-dominated})}\\
      &\propto \abs{\volmeas{U}[\nabla  \varphi\circ\varphi\inv]}\inv(\disint{\lebesgue_V}{\operatorname{id}_V}[\obs]\times \lebesgue_W) &\text{(\cref{lem:disintegration-product})}\\
      &= \abs{\volmeas{U}[\nabla  \varphi\circ\varphi\inv]}\inv(\delta_\obs\times \lebesgue_W) &\text{(\cref{lem:disintegration-bijection})}
  \end{align*}
  hence
  \begin{align*}
      \disint{\prior|_U}{\obsop}[\obs] &\propto \dv{\prior}{\volmeas{U}} \disint{\volmeas{U}}{\obsop}[\obs] &\text{(\cref{lem:disintegration-dominated})}\\
        & = \dv{\prior}{\volmeas{U}} \varphi_\star\inv\left(\disint{(\varphi_\star\volmeas{U})}{ \pi_V}[\obs]\right) &\text{(\cref{lem:disintegration-pushforward})} \\
         & \propto \dv{\prior}{\volmeas U}{} \abs{\volmeas{U}[\nabla  \varphi]}\inv  \varphi_\star\inv(\delta_{\obs} \times \lebesgue_W).
  \end{align*}
  We also have, by \cref{eq:riemann-volume-measure-pushforward}, the pushforward Riemannian volume measure on submanifold $\obsfib\cap U$ of $U$ is given by
  \begin{equation*}
      \varphi_\star\volmeas{\obsfib\cap U} =\abs{\volmeas{\obsfib\cap U}[\nabla  \varphi|_{\kernel{\jacobian{\obsop}{}}}\circ \varphi\inv]}\inv(\delta_y \times \lebesgue_W)
  \end{equation*}
  where we used that $\mathrm T_x \obsfib\cap U = \kernel{\jacobian{\obsop}(x)}$. So
  \begin{equation*}
      \volmeas{\obsfib\cap U} = \abs{\volmeas{\obsfib\cap U}[\nabla  \varphi|_{\kernel{\jacobian{\obsop}{}}}]}\inv \varphi_\star\inv(\delta_{\obs} \times \lebesgue_W).
  \end{equation*}
  Thus we see that $\disint{\prior|_U}{\obsop}[y] \ll \volmeas{\obsfib\cap U}$ and,  by \cref{eq:diff-form-product},
  \begin{align*}
      \dv{\disint{\mu|_U}{h}[\obs]}{\volmeas{\obsfib\cap U}} &\propto \dv{\prior}{\volmeas{U}}\abs{\frac{\volmeas{\obsfib\cap U}[\nabla  \varphi|_{\kernel{\jacobian{\obsop}{}}}]}{\volmeas{U}[\nabla  \varphi]}} \\
      &= \dv{\prior}{\volmeas{U}}\abs{\frac{\volmeas{\obsfib\cap U}[\nabla  \varphi|_{\kernel{\jacobian{\obsop}{}}}]}{\volmeas{\obsfib\cap U}[\nabla  \varphi|_{\kernel{\jacobian{\obsop}{}}}]\cdot \volmeas{\kernel{\jacobian{\obsop}}^\perp}[\nabla\varphi|_{\kernel{\jacobian{\obsop}}^\perp}]}} \\
      &= \dv{\prior}{\volmeas{U}} \abs{\volmeas{\kernel{\jacobian{\obsop}}^\perp}[\nabla\varphi|_{\kernel{\jacobian{\obsop}}^\perp}]}\inv.
  \end{align*}
  $x\in\obsfib$ was arbitrary, and $U$ was chosen as a neighbourhood of $x$. Moreover $\left.\volmeas{\obsfib} \right\rvert_U = \volmeas{\obsfib\cap U}$, and by \cref{lem:disintegration-restriction} $\left. \priordisint \right\rvert_U \propto \disint{\mu|_U}{h}[\obs]$. Thus, patching together open sets $U$ we get
  \begin{equation*}
      \dv{\priordisint}{\volmeas{\obsfib}} \propto
      \dv{\prior}{\volmeas{\latsp}} \cdot \abs{\volmeas{\kernel{\jacobian{\obsop}}^\perp}[\nabla\obsop]}\inv.
  \end{equation*}
\end{proof}

In the case $\latsp = \R^d$, with the standard Euclidean metric we get $\volmeas{\R^d} = \lambda$ the Lebesgue measure, and the restricted volume form term involves the determinant of the Jacobian of $h$, restricted to the orthogonal complement of its kernel. This observation yields the special case $\latsp=\R^d$ result from the introduction (\cref{thm:disint-density-Rd}).

\begin{remark} \label{rmk:disintegration-riemannian-structure}
  In \cref{thm:disintegration-density}, the disintegration does not depend on the choice of Riemannian metric. 
  Instead, the Riemannian metric provides a way to constructively represent the disintegration.
  We endow $\latsp$ with a Riemannian geometry only to provide a reference measure, $\volmeas{\obsfib}$, with-respect-to which we can then construct a density for the disintegration. Note that Riemannian metric can always be constructed on any $C^k$ manifold with $k\geq 1$ (\cref{rmk:riemann-existence}).

  Particularly pertinent for \cref{sec:rcp-weak-modes} is that if $\latsp$ has a metric structure, that structure need not have any relationship with the Riemannian distance function implied by the Riemannian metric chosen.
\end{remark}

\begin{remark}\label{rmk:disintegration-fibers}
    \Cref{thm:disintegration-density} shows that, for fixed $\obs\in\obssp$, $\priordisint$ does not merely depend on $\obsop$ through $\obsfib$, but also on the behaviour of $\obsop$ in a neighbourhood of $\obs$, through its gradient. However, comparing this with \cref{prop:disintegration-equiv-obs}, we see that it is not the \emph{values} of $h$ that affect the disintegration; \cref{prop:disintegration-equiv-obs} rather shows that $\priordisint$ only depends on $\obsop$ through the family of fibers $\{\obsop\inv(\{\obs'\})\}_{\obs'\in\obssp}$ with distinguished element $\obsfib$.

\end{remark}

\begin{example}
  \label{ex:gaussian-ellipse}
  Consider a standard Gaussian prior $\prior = \gaussian{0}{I}$ on $\latsp = \R^2$.
  Suppose we wish to condition on an exact observation $\obs \in \obssp = \R_{\ge 0}$ made through the quadratic observation operator $\obsop(\latval) = \frac{x_1^2}{a^2} + \frac{x_2^2}{b^2}$ with $a = 1$ and $b = \frac{1}{2}$.
  In this case the observation fibers $\obsfib$ are centered, axis-aligned ellipses with width $2 a \sqrt{\obs}$ and height $2 b \sqrt{\obs}$.
  The above is visualized in \cref{fig:mode_example:prior}.
  In the \namecref{fig:mode_example}, we choose a uniformly spaced grid of observations $\obs \in \obssp$ to visualize different regular conditional distributions simultaneously.
  In this setting, \cref{asm:prior-observation} is fulfilled (for $k = \infty$), which means that we can construct both the restricted measures $\prior_{\obsfib}$ from \cref{def:restricted-measure} as well as the version of the disintegration from \cref{thm:disintegration-density} (or, more specifically, \cref{thm:disint-density-Rd}), which we visualize in \cref{fig:mode_example:restriction,fig:mode_example:disintegration}, respectively.
  In these plots, every fiber uses its own independent color map, which means that colors are not comparable between fibers.
  We can observe a stark contrast between the restriction densities and the disintegration densities along the fibers closest to the origin.

  This is due to the corrective term involving the volume form in \cref{thm:disintegration-density}:
  Here, by \cref{thm:disint-density-Rd}, this term simplifies to $\norm{\nabla \obsop(\latval)}_2^{-1}$.
  We have
  \begin{equation*}
    \norm{\nabla \obsop(\latval)}_2^{-1}
    = \frac{1}{2} \sqrt{\frac{x_1^2}{a^4} + \frac{x_2^2}{b^4}}^{-1}
    = \frac{a}{2} \sqrt{\frac{x_1^2}{a^2} + \frac{x_2^2}{b^2} \cdot \frac{a^2}{b^2}}^{-1}
    \le \frac{a}{2} \sqrt{\frac{x_1^2}{a^2} + \frac{x_2^2}{b^2}}^{-1}
    = \frac{a}{2 \sqrt{\obs}},
  \end{equation*}
  since $\frac{a^2}{b^2} \ge 1$.
  By an analogous argument, $\norm{\nabla \obsop(\latval)}_2^{-1} \ge \frac{b}{2 \sqrt{\obs}}$, where equalities are attained at the intersections with the respective half-axes of the ellipse.
  This means that the corrective term is larger if $\obs$ is small, i.e., for fibers that are closer to the origin.
  Moreover, the corrective term varies in strength by a factor of $\frac{a}{b}$ between the poles of the ellipse.
  
  Since the observations in \cref{fig:mode_example} are uniformly spaced, the relative magnitude of this term can also be read off the plot from the ``density'' of the fibers in the normal direction, as formalised by the inverse-magnitude of the gradient in that direction.
  Intuitively speaking, the corrective factor accounts for the fact that the law of total probability encoded in \cref{def:disintegration}\labelcref{item:disintegration-law-of-total-measure} still requires the predictive probability mass to be distributed locally between fibers, even if those fibers lie less dense.
  This explains the discrepancy between the restriction and disintegration densities.
  For a more direct comparison of the densities, we also single out an observation at $\obs = 1.01$ and parameterize the corresponding elliptical fiber curve by arc length starting at the point and velocity vector marked in red in \cref{fig:mode_example:restriction,fig:mode_example:disintegration}.
  \Cref{fig:mode_example:densities} show the densities along this fiber curve.
  Note that, due to the arclength parameterization, no additional change of variables is needed.
  Remarkably, this panel also shows that the modes of the restriction and disintegration densities are maximally distant.
\end{example}

\section{Modes of Disintegrations}
\label{sec:rcp-weak-modes}
For probability measures $\prior$ on Euclidean spaces that have a density $\dv{\prior}{\lebesgue}$ with respect to the Lebesgue measure $\lebesgue$, the modes are defined as the maximizers of said density.
However, neither the restricted measures nor the elements of the disintegrations constructed in \cref{sec:disintegration-manifolds} have densities with respect to $\lebesgue$ in general, as they are supported on lower dimensional fibers.
Luckily, there are generalizations of this notion of mode, which measure the local accumulation of probability mass by comparing the probability of metric balls in the small noise limit.
We first recap the relevant definitions of and key results about these generalized notions of modes below.

For the remainder of the article, $\latsp$ will be a metric space and $\prior$ a locally finite measure on the Borel $\sigma$-algebra of $\latsp$.
\begin{definition}[Modes]
  \label{def:modes}
  A point $\mode \in \operatorname{supp}(\prior) \subset \latsp$ is called a
  \begin{enumerate}[label=(\roman*)]
    \item \label{def:modes:item:weak}%
      \emph{weak mode} of $\prior$ if
      \begin{equation*}
        \limsup_{r \downarrow 0} \frac{\prior(\oball{r}{\latval})}{\prior(\oball{r}{\mode})} \le 1
      \end{equation*}
      for all $\latval \in \latsp$, and a
    \item \label{def:modes:item:strong}%
      \emph{strong mode} of $\prior$ if
      \begin{equation*}
        \limsup_{r \downarrow 0} \frac{\sup_{\latval \in \latsp} \prior(\oball{r}{\latval})}{\prior(\oball{r}{\mode})} \le 1.
      \end{equation*}
  \end{enumerate}
\end{definition}

The definitions above are due to \citet[Definition 3.1]{Dashti2013MAPBayesianInverseProblems}, \citet[Definition 4]{Helin2015MaximumPosterioriProbability}, and, in the form presented here, \citet[Definition 3.6 and 3.7]{Ayanbayev2022GammaConvergence}.
However, all of the above only define modes of probability measures, so we restate definitions and results here (often without proof) if they also hold in our more general setting of locally finite measures.

The following result motivates the naming of weak and strong modes.

\begin{proposition}[{\citealp[Proposition 3.9]{Ayanbayev2022GammaConvergence}}]
  Every strong mode $\latval \in \latsp$ of $\prior$ is a weak mode of $\prior$.
\end{proposition}

Just as for modes defined via Lebesgue densities, weak modes can be characterized as the solutions of optimization problems.
However the objective function of these optimization problems needs to be defined only using properties of the measure $\prior$.
These objective functions are given by Onsager-Machlup functionals, as in the next definition.

\begin{definition}[Onsager-Machlup Functional; {\citealp[see][Definition 3.1]{Ayanbayev2022GammaConvergence}}]
  Let $\omdom \subset \operatorname{supp}(\prior) \subset \latsp$ be non-empty.
  A function $\omfctl \colon \omdom \to \R$ is called an \emph{Onsager-Machlup (OM) functional} for $\prior$ if
  \begin{equation*}
    \lim_{r \downarrow 0} \frac{\prior(\oball{r}{\latval_1})}{\prior(\oball{r}{\latval_2})}
    = \exp(\omfctl(\latval_2) - \omfctl(\latval_1))
  \end{equation*}
  for all $\latval_1, \latval_2 \in \omdom$.
  An OM functional is called \emph{exhaustive} if
  \begin{equation*}
    \lim_{r \downarrow 0} \frac{\prior(\oball{r}{\latval_1})}{\prior(\oball{r}{\latval_2})} = 0
  \end{equation*}
  for all $\latval_1 \in \latsp \setminus \omdom$ and $\latval_2 \in \omdom$.
\end{definition}
As noted by \citet{Ayanbayev2022GammaConvergence}, the exhaustiveness property ensures that the domain of definition of the OM function is in some sense ``maximal''.
\Cref{lem:exhaustive-om} is sometimes useful to verify exhaustiveness of an OM functional.

The weak modes of the measure $\prior$ can be characterized as the solutions to a minimization problem whose objective function is given by the OM functional.
Intuitively speaking, the OM functional takes the role of a negative log Legesgue density.

\begin{theorem}[{\citealp[][Proposition 4.1]{Ayanbayev2022GammaConvergence}}]
  \label{thm:omfctl-weak-modes}
  If $\omfctl \colon \omdom \to \R$ is an exhaustive OM functional for $\prior$, then $\mode$ is a weak mode of $\prior$ if and only if
  \begin{equation*}
    \mode \in \argmin_{\latval \in \omdom} \omfctl(\latval).
  \end{equation*}
\end{theorem}

We will now make use of \cref{thm:disintegration-density} to characterize the modes of the disintegration $\priordisint[]$ by constructing exhaustive OM functionals.

\begin{remark}
Our characterisation is complicated by the fact that a disintegration is only $\predictive$-almost-surely unique; in other words, there are (typically infinitely-many) \emph{versions} of a disintegration that differ on $\predictive$-null sets.
To simplify terminology we focus on the specific version of the disintegration constructed in \cref{thm:disintegration-density}, and say that a function is an (exhaustive) OM functional for the disintegration $\priordisint$ if it is an (exhaustive) OM functional for the version from \cref{thm:disintegration-density}, which is well-defined under \cref{asm:prior-observation} if $\obs$ is a regular value of $\obsop$.
Note that, under this assumption, the non-regular values of $\obsop$ form a $\predictive$-null set.
This ensures that the functions given below are still (exhaustive) OM functionals for $\predictive$-almost all $\obs \in \obssp$ in all other versions of the disintegration.
\end{remark}

The general mode theory presented thus far only required $\latsp$ to be a metric space.
However, we will now work under \cref{asm:prior-observation} in order for \cref{thm:disintegration-density} to hold.
This assumes in particular that $\latsp$ is a manifold, and we implicitly require that the manifold topology coincides with the metric space topology.
As mentioned in \cref{rmk:disintegration-riemannian-structure}, we do not generally require the Riemannian distance function (see \cref{def:riemann-distance}) to coincide with the metric space structure.
This gives us the flexibility to study \emph{non-Riemannian} metric space structures, for example the one induced by $\ell_p$ norms in \cref{sec:failure-modes}.
We will later study the special case where the metric space structure coincides with the Riemannian distance function (\cref{prop:omfctl-volmeas,cor:rcp-weak-modes-riemann}).

\begin{theorem}
  \label{thm:rcp-weak-modes}
  Under \cref{asm:prior-observation} with $k = 1$, equip $\latsp$ with a Riemannian metric.
  Assume that $\obs \in \obssp$ is a regular value and that the Riemannian restricted measure $\prior_{\obsfib}$ admits an OM functional $\omfctl[\prior_{\obsfib}] \colon \omdom[\prior_{\obsfib}] \to \R$.
  Then the disintegration $\priordisint$ admits an OM functional $\omfctl[\priordisint] \colon \omdom[\priordisint] \to \R$, with $\omdom[\priordisint] \defeq \omdom[\prior_{\obsfib}]$, and
  \begin{equation*}
    \omfctl[\priordisint](\latval)
    \defeq \omfctl[\prior_{\obsfib}](\latval) + \log \abs{\volmeas{\kernel{\jacobian{\obsop}(\latval)}^\perp}(\latval)[\nabla\obsop(\latval)]}.
  \end{equation*}
  Furthermore if $\omfctl[\prior_{\obsfib}]$ is exhaustive then $\omfctl[\priordisint]$ is exhaustive.
\end{theorem}
\begin{proof}
  By \cref{thm:disintegration-density}, the disintegration $\priordisint$ can be constructed as
  \begin{equation*}
    \priordisint(\dd{x}) \propto \abs{\volmeas{\kernel{\jacobian{\obsop}(\latval)}^\perp}(\latval)[\nabla\obsop(\latval)]}\inv \prior_{\obsfib}(\dd{x}).
  \end{equation*}
  The function
  \begin{equation*}
    \latsp \to \R_{> 0},
    \latval \mapsto \abs{\volmeas{\kernel{\jacobian{\obsop}(\latval)}^\perp}(\latval)[\nabla\obsop(\latval)]}\inv
  \end{equation*}
  is continuous.
  Hence, the result follows by \cref{lem:omfctl-reweighted}.
\end{proof}

One might wonder whether the OM functional of the Riemannian restricted measure, $\omfctl[\prior_{\obsfib}](\latval)$, is given by $-\log\dv{\prior}{\volmeas{\latsp}}$.
This fails to hold for general metrics on $\latsp$ and, perhaps surprisingly, even for non-Euclidean normed spaces $\latsp$, as we will show in \cref{sec:failure-modes}.
However, in the following we show that this is indeed the case if the metric on $\latsp$ is induced by the Riemannian metric used to define $\volmeas{\latsp}$ and under mild regularity assumptions on $\latsp$, $\dv{\prior}{\volmeas{\latsp}}$ and $\obsop$.

Note that since volume measures on manifolds are locally finite, we can discuss OM functionals $\omfctl[\volmeas{\latsp}]$ and  $\omfctl[\volmeas{\obsfib}]$ for the Riemannian volume measures $\volmeas{\latsp}$ and $\volmeas{\obsfib}$ respectively.
\begin{assumption}
  \label{asm:continuous-density}
  $\prior$ is absolutely continuous with respect to a volume measure on the manifold $\latsp$, with continuous Radon-Nikodym derivative.
\end{assumption}
When $\latsp$ is equipped with a Riemannian metric, \cref{asm:continuous-density} implies that $\prior$ is absolutely continuous with respect to $\volmeas{\latsp}$ (\cref{rmk:volume-measures}) and $\dv{\prior}{\volmeas{\latsp}}$ is continuous.
\begin{corollary}
  \label{cor:rcp-weak-modes}
  Under \cref{asm:prior-observation} with $k = 1$ and \cref{asm:continuous-density}, assume that $\obs \in \obssp$ is a regular value and $\volmeas{\obsfib}$ admits an OM functional $\omfctl[\volmeas{\obsfib}] \colon \omdom[\volmeas{\obsfib}] \to \R$.
  Then the disintegration $\priordisint$ admits an OM functional $\omfctl[\priordisint] \colon \omdom[\priordisint] \to \R$ with $\omdom[\priordisint] \defeq \omdom[\prior_{\obsfib}] \cap \operatorname{supp}^\circ(\dv{\prior}{\volmeas{\latsp}})$\footnote{We define $\operatorname{supp}^\circ(f) \defeq \set{\latval \in \latsp \where f(\latval) > 0}$.} and
  \begin{equation*}
    \omfctl[\priordisint](\latval)
    \defeq \omfctl[\volmeas{\obsfib}](\latval) - \log \dv{\prior}{\volmeas{\latsp}}(\latval){(\latval)} + \log \abs{\volmeas{\kernel{\jacobian{\obsop}(\latval)}^\perp}(\latval)[\nabla\obsop(\latval)]}.
  \end{equation*}
  If, additionally, $\omfctl[\volmeas{\obsfib}]$ is exhaustive then $\omfctl[\priordisint]$ is exhaustive.
\end{corollary}
\begin{proof}
  The Radon-Nikodym derivative $\dv{\prior}{\volmeas{\latsp}}$ is continuous and hence, by \cref{lem:omfctl-reweighted}, $$\omfctl[\prior_{\obsfib}] = \omfctl[\volmeas{\obsfib}] - \log \dv{\prior}{\volmeas{\latsp}}(\latval)$$ is an OM functional for $\prior_{\obsfib}$.
  and $\omfctl[\prior_{\obsfib}]$ is exhaustive if $\omfctl[\volmeas{\obsfib}]$ is.
  The claim then follows by \cref{thm:rcp-weak-modes}.
\end{proof}

Next we establish that, when the metric on $\latsp$ is compatible with the Riemannian geometry, the OM functional is constant.
\begin{assumption}
  \label{asm:prior-observation-C2-riemannian}
  \leavevmode
  \begin{enumerate}[label=(\roman*)]
    \item \label{item:X-C2-riemannian}
      $\latsp$ is a $d$-dimensional $C^2$ Riemannian manifold which, as a metric space, is equipped with the Riemannian distance metric.
    \item \label{item:Y-C2-observation}
      $\obssp = \R^n$ for some $n \in \N$, and $\obsop \colon \latsp \to \obssp$
      such that the points in $\latsp$ are $\prior$-almost everywhere $C^2$ regular w.r.t.~$\obsop$.
  \end{enumerate}
\end{assumption}
\begin{proposition}[Riemannian Onsager-Machlup Functionals of Volume Measures]
  \label{prop:omfctl-volmeas}
  \leavevmode
  \begin{enumerate}[label=(\roman*)]
    \item \label{item:omfctl-volmeas-latsp}
      Under \cref{asm:prior-observation-C2-riemannian}\ref{item:X-C2-riemannian}, any constant function $\omfctl[\volmeas{\latsp}] \colon \latsp \to \R$ is an exhaustive OM functional for the Riemannian volume measure $\volmeas{\latsp}$, and,
    \item \label{item:omfctl-volmeas-obsfib}
      \sloppy under \cref{asm:prior-observation-C2-riemannian}, if $\obs\in\obssp$ is a regular value, any constant function $\omfctl[\volmeas{\obsfib}] \colon \omdom[\volmeas{\obsfib}] \to \R$ with $\omdom[\volmeas{\obsfib}] = \obsfib \subset \latsp$ is an exhaustive OM functional for $\volmeas{\obsfib}$.
  \end{enumerate}
\end{proposition}

\begin{proof}
  Note that by \cref{prop:small-riemannian-balls}, the Riemannian volume of Riemannian balls is to first order
  \begin{equation*}
      \volmeas{\latsp}(B_r^\latsp(x)) = V_dr^d + o(r^d)
  \end{equation*}
  as $r\downarrow 0$, where $V_d$ is the volume of the $d$-dimensional Euclidean ball. Thus we have for $x_1, x_2\in\latsp$,
  \begin{equation*}
    \frac{\mu(B_r(x_1))}{\mu(B_r(x_2))} = \frac{V_dr^d + o(r^d)}{V_dr^d + o(r^d)} = \frac{V_d + o(1)}{V_d + o(1)} \xrightarrow{r\downarrow 0} 1
  \end{equation*}
  which concludes the proof of \ref{item:omfctl-volmeas-latsp}.

  For \ref{item:omfctl-volmeas-obsfib}, if $x\not\in\obsfib$ then $\volmeas{\obsfib}(B_r^{\latsp}(x))=0$ for all $r$ small enough. If $x\in\obsfib$, we have by \cref{prop:ambient-riemannian-balls}
  \begin{equation*}
      \volmeas{\obsfib}(B_r^{\latsp}(x)) = V_{d-n}r^{d-n} + o(r^{d-n})
  \end{equation*}
  as $r\downarrow 0$, and the proof of \ref{item:omfctl-volmeas-obsfib} follows as for \ref{item:omfctl-volmeas-latsp}.
\end{proof}

Note that combining \cref{asm:continuous-density} and \cref{asm:prior-observation-C2-riemannian} amounts to \cref{asm:prior-observation} with additional regularity requirements. These results lead to a concrete characterization of OM functionals of disintegrations:
\begin{corollary}
  \label{cor:rcp-weak-modes-riemann}
  Under \cref{asm:prior-observation-C2-riemannian,asm:continuous-density}, additionally assume that $\obs \in \obssp$ is a regular value.
  Then $\priordisint$ admits an exhaustive OM functional $\omfctl[\priordisint] \colon \omdom[\priordisint] \to \R$ with $\omdom[\priordisint] \defeq \obsfib \cap \operatorname{supp}^\circ(\dv{\prior}{\volmeas{\latsp}})$ and
  \begin{equation*}
    \omfctl[\priordisint](\latval)
    \defeq
    - \log \dv{\prior}{\volmeas{\latsp}}{(\latval)} + \log \abs{\volmeas{\kernel{\jacobian{\obsop}(\latval)}^\perp}(\latval)[\nabla\obsop(\latval)]}.
  \end{equation*}
\end{corollary}
\begin{proof}
  This follows directly from \cref{cor:rcp-weak-modes} and \cref{prop:omfctl-volmeas}\ref{item:omfctl-volmeas-obsfib}.
\end{proof}

By \cref{thm:omfctl-weak-modes}, this leads to a characterization of weak modes of the disintegration $\priordisint$ as the solutions of the constrained optimization problem
\begin{equation*}
  \begin{alignedat}{4}
    \min_{\latval \in \latsp} & \quad && - \log \dv{\prior}{\volmeas{\latsp}}{(\latval)} + \log \abs{\volmeas{\kernel{\jacobian{\obsop}(\latval)}^\perp}(\latval)[\nabla\obsop(\latval)]}, \\
    \textnormal{s.t.} &&& {\obsop}({\latval}) = {\obs}.
  \end{alignedat}
\end{equation*}
\Cref{thm:rcp-weak-modes-Rd} in the introduction is a simplified version of this result in $\R^d$.

This result shows that the maximizers of the disintegration densities in \cref{fig:mode_example:densities,fig:mode_example:densities} are indeed the weak modes of the disintegration when equipping the ambient space with the Euclidean metric.

\subsection{Restricted Modes Disagree with Modes of Disintegrations}
\label{sec:rcp-conditional-modes-vs-weak-modes}
\citet[Definition 3]{Chen2024GaussianConditionalMAP} define a variant of a mode on a fibre by restricting the search space for a strong mode from \cref{def:modes} to the fibre.
Below we generalize their definition to the setting of metric spaces and provide a weak version of the definition in order to study their relationship to weak and strong modes of disintegrations.
We will demonstrate that in general these modes differ from the modes of disintegrations and, for this reason, refer to these modes as ``restricted modes'' rather than conditional modes as in \cite{Chen2024GaussianConditionalMAP}.
\begin{definition}[Restricted Modes]
  \label{def:restricted-modes}
  Let $\obs \in \operatorname{supp}(\pushfw{\prior}{\obsop})$.
  A point $\latval^\star \in \obsfib \cap \operatorname{supp}(\prior) \subset \latsp$ is called a
  \begin{enumerate}[label=(\roman*)]
    \item \label{item:weak-restricted-mode}%
      \emph{weak $\obsfib$-restricted mode} of $\prior$ if
      \begin{equation*}
        \limsup_{r \downarrow 0} \frac{\prior(\oball{r}{\latval})}{\prior(\oball{r}{\latval^\star})} \le 1
      \end{equation*}
      for all $\latval \in \obsfib$, and a
    \item \label{item:strong-restricted-mode}%
      \emph{strong $\obsfib$-restricted mode} of $\prior$ if
      \begin{equation*}
        \limsup_{r \downarrow 0} \frac{\sup_{\latval \in \obsfib} \prior(\oball{r}{\latval})}{\prior(\oball{r}{\latval^\star})} \le 1.
      \end{equation*}
  \end{enumerate}
\end{definition}

As for regular weak and strong modes, we start by justifying the ``weak'' and ``strong'' attributes in the naming of restricted modes.

\begin{proposition}
  Every strong $\obsfib$-restricted mode $\latval \in \obsfib$ is also a weak $\obsfib$-restricted mode.
\end{proposition}
\begin{proof}
  We adapt the proof of Lemma 3.9 in \citet{Ayanbayev2022GammaConvergence}.

  Let $\latval^\star \in \obsfib \cap \operatorname{supp}(\prior)$ be a strong $\obsfib$-restricted mode of $\prior$.
  For all $\latval \in \obsfib$, we have
  \begin{equation*}
    \limsup_{r \downarrow 0} \frac{\prior(\oball{r}{\latval})}{\prior(\oball{r}{\latval^\star})}
    \le \limsup_{r \downarrow 0} \frac{\sup_{\latval' \in \obsfib} \prior(\oball{r}{\latval'})}{\prior(\oball{r}{\latval^\star})}
    \le 1,
  \end{equation*}
  since $\latval^\star$ is a strong restricted mode of $\prior$.
  Hence, $\latval^\star$ is a weak restricted mode of $\prior$.
\end{proof}

Just as for weak modes, weak restricted modes can be characterized as the minimizers of exhaustive OM functionals.
However, for weak restricted modes, we can slightly weaken the notion of exhaustiveness:
\begin{definition}
  An Onsager-Machlup functional $\omfctl \colon \omdom \to \R$ for $\prior$ is called \emph{$\obsfib$-exhaustive} if
  \begin{equation*}
    \lim_{r \downarrow 0} \frac{\prior(\oball{r}{\latval_1})}{\prior(\oball{r}{\latval_2})} = 0
  \end{equation*}
  for all $\latval_2 \in \omdom \cap \obsfib$ and $\latval_1 \in \obsfib \setminus \omdom$.
\end{definition}
We next prove the natural result that weak $\obsfib$-restricted modes are minimisers of $\obsfib$-exhaustive OM functionals.

\begin{proposition}
  \label{prop:weak-restricted-modes-omfctl}
  Let $\omfctl \colon \omdom \to \R$ be an $\obsfib$-exhaustive Onsager-Machlup functional for $\prior$.
  Then $\latval^\star$ is a weak $\obsfib$-restricted mode of $\prior$ if and only if
  \begin{equation*}
    \latval^\star \in \argmin_{\latval \in \obsfib} \omfctl(\latval).
  \end{equation*}
\end{proposition}
\begin{proof}
  We adapt the proof of Proposition 4.1 in \citet{Ayanbayev2022GammaConvergence}.

  Let $\latval \in \obsfib \cap \omdom$.
  Then, by the definition of an $\obsfib$-exhaustive OM functional, we have
  \begin{equation*}
    \limsup_{r \downarrow 0} \frac{\prior(\oball{r}{\latval'})}{\prior(\oball{r}{\latval})}
    = \lim_{r \downarrow 0} \frac{\prior(\oball{r}{\latval'})}{\prior(\oball{r}{\latval})}
    = \begin{cases*}
      \exp(\omfctl(\latval) - \omfctl(\latval')) & if $\latval' \in \omdom$, \\
      0 & otherwise,
    \end{cases*}
  \end{equation*}
  for all $\latval' \in \obsfib$.
  Hence, $\latval$ is a weak $\obsfib$-restricted mode of $\prior$ if and only if $\exp(\omfctl(\latval) - \omfctl(\latval')) \le 1$ for all $\latval' \in \obsfib$, or, equivalently, $\omfctl(\latval) \le \omfctl(\latval')$ for all $\latval' \in \obsfib$.

  Now let $\latval \in \obsfib \setminus \omdom$.
  Since $\omfctl(\latval) = +\infty$, it remains to show that $\latval$ is not a weak $\obsfib$-restricted mode of $\prior$.
  Suppose for the sake of contradiction that this is the case.
  By definition, we then have $\latval \in \operatorname{supp}(\prior)$.
  Pick $\latval' \in \obsfib \cap \omdom \subset \operatorname{supp}(\prior)$.
  Then $\frac{\prior(\oball{r}{\latval'})}{\prior(\oball{r}{\latval})} > 0$ for every $r > 0$, and hence
  \begin{equation*}
    1
    \ge \limsup_{r \downarrow 0} \frac{\prior(\oball{r}{\latval'})}{\prior(\oball{r}{\latval})}
    = \left( \liminf_{r \downarrow 0} \frac{\prior(\oball{r}{\latval})}{\prior(\oball{r}{\latval'})} \right)\inv,
  \end{equation*}
  which is equivalent to
  \begin{equation*}
    \liminf_{r \downarrow 0} \frac{\prior(\oball{r}{\latval})}{\prior(\oball{r}{\latval'})} \ge 1.
  \end{equation*}
  This contradicts the $\obsfib$-exhaustiveness of the OM functional.
  We conclude that $\obsfib \setminus \omdom$ does not contain any weak $\obsfib$-restricted modes of $\prior$.
\end{proof}

Using some of the assumptions and results from the previous sections, we can build $\obsfib$-exhaustive OM functionals for $\prior$:
\begin{corollary}
  \label{cor:weak-restricted-modes-weak-modes-restriction}
  Under \cref{asm:prior-observation-C2-riemannian,asm:continuous-density}, if $\obs\in\obssp$ is a regular value, the function $\restrict{- \log \dv{\prior}{\volmeas{\latsp}}}{\obsfib}{}$ is
  \begin{enumerate}[label=(\roman*)]
    \item \label{item:restricted-om-prior}%
      an $\obsfib$-exhaustive OM functional for $\prior$, and,
    \item \label{item:om-restricted-measure}%
      an exhaustive OM functional for $\prior_{\obsfib}$.
  \end{enumerate}
\end{corollary}
\begin{proof}
  \leavevmode
  Since $- \log \dv{\prior}{\volmeas{\latsp}}$ is continuous, \ref{item:restricted-om-prior} follows from \cref{prop:omfctl-volmeas}\ref{item:omfctl-volmeas-latsp}, and \ref{item:om-restricted-measure} follows from \cref{prop:omfctl-volmeas}\ref{item:omfctl-volmeas-obsfib}, by the same argument as in the proof of \cref{thm:rcp-weak-modes}.
\end{proof}

\Cref{cor:weak-restricted-modes-weak-modes-restriction} together with \cref{thm:omfctl-weak-modes,prop:weak-restricted-modes-omfctl} imply that, under \cref{asm:prior-observation-C2-riemannian,asm:continuous-density}, restricted modes correspond to modes of the Riemannian restricted measure, \emph{not} of the disintegration.
This explains our choice to rename this notion of mode from ``conditional mode'' to ``restricted mode''.

In \cref{fig:mode_example} and \cref{ex:gaussian-ellipse} we have seen that weak restricted modes and weak modes of the disintegration can disagree catastrophically, even in finite-dimensional Hilbert spaces.

\subsection{Failure Modes in Normed Spaces}
\label{sec:failure-modes}
Many of the results above apply when the metric on $\latsp$ coincides with the Riemannian distance function.
If the underlying space $\latsp$ is a vector space, then this is the case when $\latsp$ is equipped with an inner product that generates a norm and a metric on the space.
However, especially in probabilistic numerical methods, it is often more natural to work instead in a normed space whose norm is not generated by an inner product.
For instance, when approximating strong solutions to partial differential equations, it is common to choose $\latsp$ to be a Hölder space with the appropriate norm.

It is now natural to ask whether a change in norm affects the results on weak modes of disintegrations and restricted measures above.
Below we will show that this is indeed the case.
It is well-known that Onsanger-Machlup functionals have a strong dependence on the metric used to define them \citep[see e.g.,][Example B.4]{Ayanbayev2022GammaConvergence}.
However, perhaps counterintuitively, in our setting the differing behaviour of the weak modes under different norms only arises after disintegrating or restricting, leaving the prior's modes unaffected.

In this subsection $\latsp=\R^d$ is a normed space. Write $\lebesgue$ for the Lebesgue measure on $\R^d$.
The standard Euclidean inner product on $\R^d$ can be viewed as a Riemannian metric, with Riemannian volume form $\lebesgue$.
Thus, under \cref{asm:prior-observation}, we have a canonical restriction $\lebesgue_{\obsfib}$ of $\lambda$ to the fiber $\obsfib\subset \R^d$.
We will start by providing a version of \cref{prop:omfctl-volmeas} in normed spaces.
\begin{proposition}[Onsager-Machlup Functionals of the Lebesgue Measure in Normed Spaces]
  \label{prop:omfctl-volmeas-normed}
  \leavevmode
  \begin{enumerate}[label=(\roman*)]
    \item \label{item:om-lebesgue-vol} Any constant function $\omfctl[\lebesgue] \colon \R^d \to \R$ is an exhaustive OM functional for $\lebesgue$, and,
    \item \label{item:om-restricted-lebesgue-vol} under \cref{asm:prior-observation} with $k = 2$, if $\obs\in\obssp$ is a regular value
    then $\lebesgue_{\obsfib}$ admits an exhaustive OM functional 
    \[\omfctl[\lebesgue_{\obsfib}] \colon \omdom[\lebesgue_{\obsfib}] \to \R,\quad x \mapsto -\log V_{d-n}(x)
    \] 
    where $\omdom[\lebesgue_{\obsfib}] = \obsfib \subset \latsp$ and $V_{d-n}(x)$ is the $(d-n)$-dimensional volume of $B_1^\latsp(0)\cap \kernel{\jacobian{{\obsop}}({\latval})}$ in $\kernel{\jacobian{{\obsop}}({\latval})}\subset \R^d$.
  \end{enumerate}
\end{proposition}
\begin{proof}
    \ref{item:om-lebesgue-vol} is a consequence of the translation invariance of $\lambda$ and of the metric (induced by the norm) on $\R^d$.

    For \ref{item:om-restricted-lebesgue-vol}, if $x\not\in\obsfib$ then $\lebesgue_{\obsfib}(B_r^{\R^d}(x))=0$ for all $r$ small enough. If $x\in \obsfib$, we have by \cref{prop:ambient-normed-balls}
    \begin{equation*}
        \lebesgue_{\obsfib}(B_r^{\R^d
        }(x)) = V_{d-n}(x)r^{d-n}+o(r^{d-n})
    \end{equation*}
    as $r\downarrow 0$, and the proof of \ref{item:om-restricted-lebesgue-vol} follows as the proof of \cref{prop:omfctl-volmeas}.
\end{proof}
\Cref{prop:omfctl-volmeas-normed}\ref{item:om-lebesgue-vol} implies that the OM functional of $\prior$, and hence its modes, do not depend on the norm on $\R^d$. 
However, surprisingly, \cref{prop:omfctl-volmeas-normed}\ref{item:om-restricted-lebesgue-vol} implies that the OM functional and the modes of $\priordisint$ does depend on the norm. 
Intuitively, this is because the balls $B^\latsp_r(0)$ are not necessarily isotropic, and so the OM functional depends on the orientation of the fibre with respect to the balls.

Combining \cref{prop:omfctl-volmeas-normed} with \cref{cor:rcp-weak-modes}, we obtain the explicit characterization of the OM functional of the disintegration in normed spaces:
\begin{corollary}
  Let $\latsp = \R^d$ a normed space, and suppose \cref{asm:prior-observation}\ref{item:Y-observation} holds with $k = 2$, as well as \cref{asm:continuous-density}. If $\obs \in \obssp$ is a regular value, then $\priordisint$ admits an exhaustive OM functional $\omfctl[\priordisint] \colon \omdom[\priordisint] \to \R$ with $\omdom[\priordisint] \defeq \obsfib \cap \operatorname{supp}^\circ(\dv{\prior}{\lebesgue})$ and
  \begin{equation*}
    \omfctl[\priordisint](\latval)
    \defeq
    -\log V_{d-n}(x)- \log \dv{\prior}{\lebesgue}{(\latval)} + \log \abs{\det \jacobian{{\obsop}}({\latval}) \vert_{\kernel{\jacobian{{\obsop}}({\latval})}^\perp}},
  \end{equation*}
  where $V_{d-n}(0)$ is the $(d-n)$-dimensional volume of $B_1^\latsp(0)\cap \kernel{\jacobian{{\obsop}}({\latval})}$ in $\kernel{\jacobian{{\obsop}}({\latval})}\subset \R^d$.
\end{corollary}

\begin{figure}
  \subcaptionbox{Riemannian Restricted Measure}[0.497\linewidth]{%
    \includegraphics[width=\linewidth]{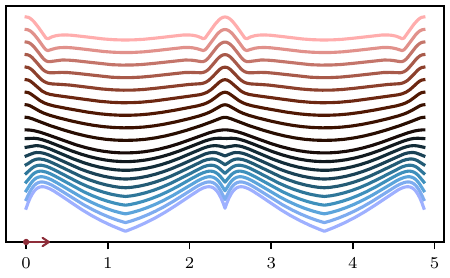}
  }
  \subcaptionbox{Disintegration}[0.497\linewidth]{%
    \includegraphics[width=\linewidth]{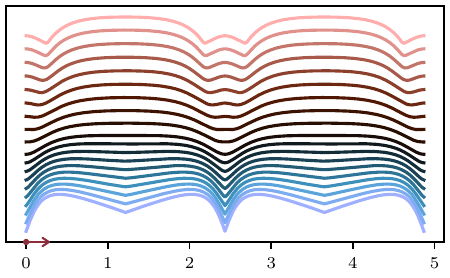}
  }%
  \caption{%
    Onsager-Machlup functionals of the restriction and disintegration in \cref{ex:gaussian-ellipse,fig:mode_example:densities} with respect to different $\ell_p$ norms.
    The bottom (blue) lines correspond to $p = 1$, the central (black) lines correspond to $p = 2$, and the top (salmon) lines correspond to $p = \infty$.
    The remaining values interpolate $\atan$-uniformly between these values. Each of these OM functionals can be shifted by an arbitrary constant, so we spread them vertically for visualization purposes.
    The change in norm leads to OM functionals with different minimizers, i.e., different weak modes of the distributions.
    For instance, for the disintegration, the $\ell_2$ OM functional has two local minima, which correspond to local maxima of the $\ell_\infty$ OM functional. The latter admits four other local minima.
  }
  \label{fig:gaussian-ellipse:lp-om-fctls}
\end{figure}

The counterintuitive behaviour of the OM functionals with respect to a change in norm is illustrated in \cref{fig:gaussian-ellipse:lp-om-fctls}.
Here varying the value of $p$ among $\ell_p$ norms on $\R^2$ leads to different locations and numbers of minimizers, and minima of certain OM functionals become local maxima in others.

\section{Discussion}
\label{sec:discussion}
This work has pointed out multiple troubling discrepancies between disintegrations and restrictions of measures, with a particular focus on their modes.
Whether restricted modes or disintegration modes are the object of interest is problem-specific.
The dividing line between the two approaches is whether the law of total probability is obeyed.
Consequently, the paradigm selected hinges on the relevance of said law in a given application.
If all potential outcomes $\obs$ of the predictive distribution $\predictive$ are equally valid measurements, the disintegration approach seems more natural.
For instance, this is the case in most of Bayesian statistics.

However for some downstream applications, only one ``observation'' is relevant or even semantically meaningful, and therefore a restriction rather than a disintegration may be preferred.
This is the case for the Laplace approximation under function-space priors proposed by \citet{Cinquin2024FSPLaplace}, and one could argue also for certain probabilistic numerical methods, particularly those based on solution of ODEs (e.g.\ in the formalism of \cite{Tronarp2019}) for which one specifies a residual, and it is not clear that setting said residual to any other value than 0 is meaningful.
In such cases there is no clear justification for enforcing the law of total probability, as marginalising over all possible observations in $\operatorname{supp}(\predictive)$ gives weight to these irrelevant observations.
This suggests that, in such cases, a potential generalisation of the restriction operation to more general settings (such as infinite-dimensional vector spaces) is more appropriate than disintegration.

\section*{Acknowledgments}
The authors thank Philipp Hennig and Bálint Mucsányi for helpful discussions and feedback on this work.
MP and NDC gratefully acknowledge financial support by the European Research Council through ERC StG Action 757275 / PANAMA and ERC CoG Action 101123955 / ANUBIS; the DFG Cluster of Excellence “Machine Learning - New Perspectives for Science”, EXC 2064/1, project number 390727645; the German Federal Ministry of Education and Research (BMBF) through the Tübingen AI Center (FKZ: 01IS18039A); the DFG SPP 2298 (Project HE 7114/5-1); and the Carl Zeiss Foundation (project "Certification and Foundations of Safe Machine Learning Systems in Healthcare"); as well as funds from the Ministry of Science, Research and Arts of the State of Baden-Württemberg. NDC acknowledges the support of the Fonds National de la Recherche, Luxembourg. The authors thank the International Max Planck Research School for Intelligent Systems (IMPRS-IS) for supporting MP and NDC.
JC is supported by EPSRC grant EP/Y001028/1.

\bibliography{main}

\begin{thebibliography}{24}
\providecommand{\natexlab}[1]{#1}
\providecommand{\url}[1]{\texttt{#1}}
\expandafter\ifx\csname urlstyle\endcsname\relax
  \providecommand{\doi}[1]{doi: #1}\else
  \providecommand{\doi}{doi: \begingroup \urlstyle{rm}\Url}\fi

\bibitem[Ayanbayev et~al.(2022)Ayanbayev, Klebanov, Lie, and Sullivan]{Ayanbayev2022GammaConvergence}
Birzhan Ayanbayev, Ilja Klebanov, Han~Cheng Lie, and T.~J. Sullivan.
\newblock {$\Gamma$}-convergence of {O}nsager-{M}achlup functionals. {P}art {I}: With applications to maximum a posteriori estimation in {B}ayesian inverse problems.
\newblock \emph{Inverse Problems}, 38\penalty0 (2), 2022.
\newblock \doi{10.1088/1361-6420/ac3f81}.

\bibitem[Bishop(2006)]{bishop_pattern_2006}
Christopher~M. Bishop.
\newblock \emph{Pattern recognition and machine learning}.
\newblock Information science and statistics. Springer, New York, 2006.
\newblock ISBN 978-0-387-31073-2.

\bibitem[Chang and Pollard(1997)]{chang_conditioning_1997}
J.~T. Chang and D.~Pollard.
\newblock Conditioning as disintegration.
\newblock \emph{Statistica Neerlandica}, 51\penalty0 (3):\penalty0 287--317, November 1997.
\newblock ISSN 0039-0402, 1467-9574.
\newblock \doi{10.1111/1467-9574.00056}.
\newblock URL \url{https://onlinelibrary.wiley.com/doi/10.1111/1467-9574.00056}.

\bibitem[Chen et~al.(2021)Chen, Hosseini, Owhadi, and Stuart]{chen_solving_2021}
Yifan Chen, Bamdad Hosseini, Houman Owhadi, and Andrew~M. Stuart.
\newblock Solving and learning nonlinear {PDEs} with {Gaussian} processes.
\newblock \emph{Journal of Computational Physics}, 447:\penalty0 110668, December 2021.
\newblock ISSN 0021-9991.
\newblock \doi{10.1016/j.jcp.2021.110668}.
\newblock URL \url{https://linkinghub.elsevier.com/retrieve/pii/S0021999121005635}.
\newblock Publisher: Elsevier BV.

\bibitem[Chen et~al.(2024)Chen, Hosseini, Owhadi, and Stuart]{Chen2024GaussianConditionalMAP}
Yifan Chen, Bamdad Hosseini, Houman Owhadi, and Andrew~M. Stuart.
\newblock {Gaussian} measures conditioned on nonlinear observations: Consistency, {MAP} estimators, and simulation, 2024.

\bibitem[Cinquin et~al.(2024)Cinquin, Pf\"ortner, Fortuin, Hennig, and Bamler]{Cinquin2024FSPLaplace}
Tristan Cinquin, Marvin Pf\"ortner, Vincent Fortuin, Philipp Hennig, and Robert Bamler.
\newblock {FSP-Laplace}: Function-space priors for the {Laplace} approximation in {Bayesian} deep learning.
\newblock In A.~Globerson, L.~Mackey, D.~Belgrave, A.~Fan, U.~Paquet, J.~Tomczak, and C.~Zhang, editors, \emph{Advances in Neural Information Processing Systems}, volume~37, pages 13897--13926. Curran Associates, Inc., 2024.
\newblock \doi{10.48550/arxiv.2407.13711}.

\bibitem[Clason et~al.(2019)Clason, Helin, Kretschmann, and Piiroinen]{clason_generalized_2019}
Christian Clason, Tapio Helin, Remo Kretschmann, and Petteri Piiroinen.
\newblock Generalized {Modes} in {Bayesian} {Inverse} {Problems}.
\newblock \emph{SIAM/ASA Journal on Uncertainty Quantification}, 7\penalty0 (2):\penalty0 652--684, January 2019.
\newblock ISSN 2166-2525.
\newblock \doi{10.1137/18m1191804}.
\newblock URL \url{https://epubs.siam.org/doi/10.1137/18M1191804}.
\newblock Publisher: Society for Industrial \& Applied Mathematics (SIAM).

\bibitem[Cockayne et~al.(2019)Cockayne, Oates, Sullivan, and Girolami]{cockayne_bayesian_2019}
Jon Cockayne, Chris~J. Oates, T.~J. Sullivan, and Mark Girolami.
\newblock Bayesian {Probabilistic} {Numerical} {Methods}.
\newblock \emph{SIAM Review}, 61\penalty0 (3):\penalty0 756--789, January 2019.
\newblock ISSN 0036-1445, 1095-7200.
\newblock \doi{10.1137/17M1139357}.
\newblock URL \url{https://epubs.siam.org/doi/10.1137/17M1139357}.

\bibitem[Dashti et~al.(2013)Dashti, Law, Stuart, and Voss]{Dashti2013MAPBayesianInverseProblems}
M.~Dashti, K.~J.~H. Law, A.~M. Stuart, and J.~Voss.
\newblock {MAP} estimators and their consistency in {Bayesian} nonparametric inverse problems.
\newblock \emph{Inverse Problems}, 29\penalty0 (9), 2013.
\newblock \doi{10.1088/0266-5611/29/9/095017}.

\bibitem[Dellacherie and Meyer(1978)]{dellacherie_probabilities_1978}
Claude Dellacherie and Paul~André Meyer.
\newblock \emph{Probabilities and potential}.
\newblock Number~29 in North-{Holland} mathematics studies. Hermann [u.a.], Paris, 1978.
\newblock ISBN 978-2-7056-5857-1 978-0-7204-0701-3.

\bibitem[Diaconis et~al.(2013)Diaconis, Holmes, and Shahshahani]{diaconis_sampling_2013}
Persi Diaconis, Susan Holmes, and Mehrdad Shahshahani.
\newblock Sampling from a {Manifold}.
\newblock In \emph{Advances in {Modern} {Statistical} {Theory} and {Applications}: {A} {Festschrift} in honor of {Morris} {L}. {Eaton}}, volume~10, pages 102--126. Institute of Mathematical Statistics, January 2013.
\newblock \doi{10.1214/12-IMSCOLL1006}.
\newblock URL \url{https://projecteuclid.org/ebooks/institute-of-mathematical-statistics-collections/Advances-in-Modern-Statistical-Theory-and-Applications--A-Festschrift/chapter/Sampling-from-a-Manifold/10.1214/12-IMSCOLL1006}.

\bibitem[Gallot et~al.(2004)Gallot, Hulin, and Lafontaine]{Gallot2004RiemannianGeometry}
Sylvestre Gallot, Dominique Hulin, and Jacques Lafontaine.
\newblock \emph{Riemannian {Geometry}}.
\newblock Universitext. Springer, Berlin, Heidelberg, 2004.
\newblock ISBN 978-3-540-20493-0 978-3-642-18855-8.
\newblock \doi{10.1007/978-3-642-18855-8}.
\newblock URL \url{http://link.springer.com/10.1007/978-3-642-18855-8}.

\bibitem[Helin and Burger(2015)]{Helin2015MaximumPosterioriProbability}
Tapio Helin and Martin Burger.
\newblock Maximum a posteriori probability estimates in infinite-dimensional {{Bayesian}} inverse problems.
\newblock \emph{Inverse Problems}, 31\penalty0 (8), 2015.
\newblock \doi{10.1088/0266-5611/31/8/085009}.

\bibitem[Klebanov et~al.(2025)Klebanov, Lambley, and Sullivan]{klebanov_classification_2025}
Ilja Klebanov, Hefin Lambley, and T.~J. Sullivan.
\newblock Classification of small-ball modes and maximum a posteriori estimators, March 2025.
\newblock URL \url{http://arxiv.org/abs/2306.16278}.
\newblock arXiv:2306.16278 [math].

\bibitem[Lambley(2023)]{Lambley2023StrongMAP}
Hefin Lambley.
\newblock Strong maximum a posteriori estimation in {B}anach spaces with {G}aussian priors.
\newblock \emph{Inverse Problems}, 39\penalty0 (12), 2023.
\newblock \doi{10.1088/1361-6420/ad07a4}.

\bibitem[Laurent et~al.(2024)Laurent, Aldea, and Franchi]{Laurent2024SymmetryAware}
Olivier Laurent, Emanuel Aldea, and Gianni Franchi.
\newblock A symmetry-aware exploration of {Bayesian} neural network posteriors.
\newblock In \emph{ICLR}, 2024.
\newblock URL \url{https://openreview.net/forum?id=FOSBQuXgAq}.

\bibitem[Lie and Sullivan(2018)]{Lie2018EquivalenceWeakStrong}
Han~Cheng Lie and T.~J. Sullivan.
\newblock Equivalence of weak and strong modes of measures on topological vector spaces.
\newblock \emph{Inverse Problems}, 34\penalty0 (11), 2018.
\newblock \doi{10.1088/1361-6420/aadef2}.

\bibitem[Mardia and Jupp(1999)]{Mardia1999}
Kanti~V. Mardia and Peter~E. Jupp.
\newblock \emph{Directional Statistics}.
\newblock Wiley, January 1999.
\newblock ISBN 9780470316979.
\newblock \doi{10.1002/9780470316979}.
\newblock URL \url{http://dx.doi.org/10.1002/9780470316979}.

\bibitem[Possobon and Rodrigues(2022)]{possobon_geometric_2022}
Renata Possobon and Christian~S. Rodrigues.
\newblock Geometric properties of disintegration of measures, February 2022.
\newblock URL \url{https://ui.adsabs.harvard.edu/abs/2022arXiv220204511P}.
\newblock ADS Bibcode: 2022arXiv220204511P.

\bibitem[Tjur(1975)]{tjur_constructive_1975}
Tue Tjur.
\newblock \emph{A {Constructive} {Definition} of {Conditional} {Distributions}}.
\newblock Institute of Mathematical Statistics, University of Copenhagen, 1975.
\newblock Google-Books-ID: CeP7XwAACAAJ.

\bibitem[Tjur(1980)]{tjur_probability_1980}
Tue Tjur.
\newblock \emph{Probability based on {Radon} measures}.
\newblock Wiley series in probability and mathematical statistics. J. Wiley, Chichester [Eng.]; New York, 1980.
\newblock ISBN 978-0-471-27824-5.

\bibitem[Tronarp et~al.(2019)Tronarp, Kersting, S\"{a}rkk\"{a}, and Hennig]{Tronarp2019}
Filip Tronarp, Hans Kersting, Simo S\"{a}rkk\"{a}, and Philipp Hennig.
\newblock Probabilistic solutions to ordinary differential equations as nonlinear {Bayesian} filtering: a new perspective.
\newblock \emph{Statistics and Computing}, 29\penalty0 (6):\penalty0 1297–1315, September 2019.
\newblock ISSN 1573-1375.
\newblock \doi{10.1007/s11222-019-09900-1}.
\newblock URL \url{http://dx.doi.org/10.1007/s11222-019-09900-1}.

\bibitem[Tronarp et~al.(2021)Tronarp, Särkkä, and Hennig]{Tronarp2023BayesianODEMAP}
Filip Tronarp, Simo Särkkä, and Philipp Hennig.
\newblock Bayesian {ODE} solvers: the maximum a posteriori estimate.
\newblock \emph{Statistics and Computing}, 31\penalty0 (3), May 2021.
\newblock ISSN 0960-3174, 1573-1375.
\newblock \doi{10.1007/s11222-021-09993-7}.
\newblock URL \url{https://link.springer.com/10.1007/s11222-021-09993-7}.
\newblock Publisher: Springer Science and Business Media LLC.

\bibitem[Wiese et~al.(2023)Wiese, Wimmer, Papamarkou, Bischl, G{\"u}nnemann, and R{\"u}gamer]{Wiese2023SymmetryBayesianDL}
Jonas~Gregor Wiese, Lisa Wimmer, Theodore Papamarkou, Bernd Bischl, Stephan G{\"u}nnemann, and David R{\"u}gamer.
\newblock Towards efficient mcmc sampling in {Bayesian} neural networks by exploiting symmetry.
\newblock In \emph{Machine Learning and Knowledge Discovery in Databases: Research Track}, pages 459--474, Cham, 2023. Springer Nature Switzerland.
\newblock ISBN 978-3-031-43412-9.

\end{thebibliography}

\newpage
\appendix

\section*{Appendix}

\section{Differential Geometry Background} \label{app:differential_geometry}
Here we include the necessary background in differential and Riemannian geometry. For a more complete exposition to this background, we refer to \cite{Gallot2004RiemannianGeometry}.

Let $\latsp$ be a $d$-dimensional $C^k$ manifold, with $k\geq 1$.
\begin{definition}[Vector Bundle]
    For $l\in \N_0$, a $l$\textit{-vector bundle} over $\latsp$ is a $C^k$-manifold $\mathbb E$ with a $C^k$ map $\pi\colon \mathbb E \to \latsp$ such that for each $\latval \in \latsp$, $\pi\inv(\{\latval\}) = E_\latval$ is a $l$-dimensional real vector space, and there exists an open neighbourhood $U$ of $\latval$ in $\latsp$ and a diffeomorphism $\varphi\colon \pi\inv(U) \to U\times \R^l$ such that the following diagram commutes
    \[\begin{tikzcd}
    	{\pi\inv(U)} & {U\times \R^l} \\
    	U
    	\arrow["\varphi", from=1-1, to=1-2]
    	\arrow["\pi"', from=1-1, to=2-1]
    	\arrow["{\pi_U}", from=1-2, to=2-1]
    \end{tikzcd}\]
    where $\pi_U\colon U\times \R^l\to U$ is the projection onto $U$, and $\varphi|_{E_x}\colon E_x \to \{x\}\times \R^l$ is a linear isomorphism.
\end{definition}
A vector bundle should be thought of as attaching a vector space at each point of the manifold. For instance the tangent bundle $\mathrm T\latsp$ and the cotangent bundle $\mathrm T^\star\latsp$ are vector bundles over $\latsp$.
\begin{definition}[Section]
    A \textit{section} $s$ of a vector bundle $\mathbb E$ over $\latsp$, is a $C^k$ map $s\colon X\to \mathbb E$ such that $\pi\circ s = \operatorname{id}_\latsp$. We write $\Gamma(\mathbb E)$ for the space of sections over $\mathbb E$.
\end{definition}
This means $s(x)$ is a `vector' at $x$. For example, given $f\colon \latsp \to \R$ we have $\jacobian{f} \in \Gamma(\mathrm T^\star\latsp)$.
\begin{definition}[Tensor]
    For $(p,q)\in \N^2_0$, a $(p,q)$\textit{-tensor} $\alpha(x)$ at $x\in \latsp$ is a multilinear map
    \begin{equation*}
        \alpha\colon \underbrace{\mathrm T^\star_x\latsp \times \dots \times \mathrm T^\star_x\latsp}_{p \text{ times}} \times \underbrace{\mathrm T_x\latsp\times \dots \times \mathrm T_x\latsp}_{q \text{ times}} \to \R,
    \end{equation*}
    i.e.~for each $1\leq j \leq p$ and $ w_1,\dots,  w_p,  w'_j \in \mathrm T_x^\star\latsp$, $ v_1,\dots,  v_q,\in \mathrm T_x\latsp$, $a_j, b_j \in \R$,
    \begin{equation*}
        \alpha(x)[ w_1,\dots, a_j w_j+b_j w'_j,\dots,  v_q] = a_j\alpha(x)[ w_1,\dots,  w_j, \dots,  v_q] + b_j\alpha(x)[ w_1,\dots,  w'_j, \dots,  v_q]
    \end{equation*}
    and similarly in the $ v_j$ arguments. Note that we can also view a $(p,q)$ tensor as an element $\alpha(x)\in \underbrace{\mathrm T_x\latsp \otimes\dots \otimes \mathrm T_x\latsp}_{p \text{ times}} \otimes \underbrace{\mathrm T_x^\star\latsp \otimes\dots \otimes \mathrm T_x^\star\latsp}_{q \text{ times}}=: (\mathrm T_x\latsp)^{\otimes p} \otimes (\mathrm T_x^\star\latsp)^{\otimes q}$.\\
    This provides a natural definition of a $(p,q)$\textit{-tensor field} $\alpha$ as an element $\alpha \in \Gamma((\mathrm T_x\latsp)^{\otimes p} \otimes (\mathrm T_x^\star\latsp)^{\otimes q})$. For $x\in\latsp$, $\alpha(x)$ is thought of as a multilinear map as above.
\end{definition}
\begin{definition}[Pullback]
    If $\obssp$ is another $C^k$ manifold and $f\colon \latsp\to \obssp$ is a $C^k$ map, we can \textit{pullback} a $(0,q)$-tensor $\alpha(f(x))$ at $f(x)\in\obssp$ to obtain a $(0,q)$-tensor $f^\star\alpha(x)$ at $x\in \latsp$ given by
    \begin{equation*}
        f^\star\alpha(x)[ v_1,\dots, v_q] := \alpha(f(x))[\jacobian{f}(x)[ v_1],\dots,\jacobian{f}(x)[ v_q]]
    \end{equation*}
    for $ v_1,\dots, v_q\in \mathrm T_x \latsp$.\\
    If $f$ is a injective then using the above construction we have a well-defined notion of a pullback of a $(0,q)$-tensor field $\alpha$ on $\obssp$ to a $(0,q)$-tensor field $f^\star\alpha$ on $\latsp$.
\end{definition}
\begin{definition}[Differential Form]\label{def:differential-form}
    For $q\in\N_0$, a $q$\textit{-form} is a $(0,q)$-tensor field which is antisymmetric, i.e.~for each $1\leq j < q$, $x\in\latsp$ and $ v_1,\dots, v_q \in \mathrm T_x \latsp$,
    \begin{equation*}
        \omega(x)[ v_1,\dots,  v_j,  v_{j+1},\dots,  v_q] = -\omega(x)[ v_1,\dots,  v_{j+1},  v_j,\dots,  v_q].
    \end{equation*}
\end{definition}
\begin{definition}[Volume Measure]\label{def:volume-measure}
    Any $d$-form $\omega$ defines a \emph{volume measure} on $\latsp$ which, by abuse of notation, we also denote by $\omega$. This is the unique measure such that, for any chart $\varphi \colon U \to V$, where $U\subset \latsp$, $V\subset \R^d$, $\varphi_\star (\omega|_U) \ll \lebesgue_V$ the Lebesgue measure on $V$, with
    \begin{equation}\label{eq:volume-measure}
        \dv{\varphi_\star (\omega|_U)}{\lebesgue_V} {( z)} = \abs{\varphi^{-1\star}\omega( z)[ e_1,\dots,  e_d]}
    \end{equation}
    where $ e_1,\dots,  e_d \in \R^d$ are the standard unit vectors.
\end{definition}
\begin{definition}[Orientability]
    $\latsp$ is \emph{orientable} if it admits a nowhere vanishing $d$-form.
\end{definition}
\begin{definition}[Regular Point{\slash}Value]\label{def:regular-point}
    Let $f\colon \latsp \to \obssp$ be a set-theoretic map between manifolds. $\latval\in\latsp$ is a \textit{regular point} of $f$ if $f$ is $C^k$ in a neighbourhood of $x$ and $\jacobian{f}(x)\colon T_x\latsp \to T_{f(x)}\obssp$ is surjective. $\obs\in\obssp$ is a \textit{regular value} of $f$ if each $x\in f\inv(\{y\})$ is a regular point.
\end{definition}

\begin{theorem}[Implicit Function{\slash}Preimage Theorem]\label{thm:implicit-function}
    Let $f\colon \latsp \to \obssp$ be a set-theoretic map between manifolds. Write $n :=\dim \obssp$. If $x\in\latsp$ is a regular point of $f$, then there exists neighbourhoods $U$ of $x$ in $\latsp$, $V$ of $y$ in $\obssp$, $W\subset\R^{d-n}$ and a $C^k$ diffeomorphism $\varphi\colon U\to V\times W$ such that $\pi_V \circ \varphi = f$, where $\pi_V\colon V\times W \to V$ is the projection

    Moreover, if $y\in \obssp$ is a regular value of $f$, then $f\inv(\{y\})$ is a $C^k$ submanifold of $\latsp$.
\end{theorem}

\subsection{Riemannian Geometry}
\begin{definition}[Riemannian Metric]
    A \textit{Riemannian metric} $g$ is a $(0,2)$-tensor field which is positive definite, i.e.~for each $x\in \latsp$ and each $ v\in T_x\latsp\setminus \{0\}$, $g(x)[ v, v] >0$. It defines an inner product, and hence a norm $\|\cdot\|$, on each tangent space.
\end{definition}
\begin{remark}\label{rmk:riemann-existence}
    A Riemannian metric can be constructed on any $C^k$ manifold with $k\geq 1$. This can be done by choosing an inner product in each coordinate patch and smoothly ``sticking'' them together with a partition of unity.
\end{remark}
In the remainder of this appendix, $\latsp$ is a $d$-dimensional $C^k$ Riemannian manifold equipped with a Riemannian metric $g$.
\begin{definition}[Gradient]
    For a $C^k$ map $f\colon \latsp\to \R^n$, we define its \textit{gradient} $\nabla f \in \Gamma((T\latsp)^n)$ as, for each $1\leq j\leq n$ and $x\in\latsp$, $\nabla f (x)_j$ is the unique vector in $T_x\latsp$ such that
    \begin{equation*}
        g(\nabla f(x)_j,\cdot) = \pi_j\circ \jacobian{f}(x)
    \end{equation*}
    where $\pi_j\colon \R^n \to \R$ is the projection onto the $j$\textsuperscript{th} coordinate.
\end{definition}
\begin{definition}[Riemannian Volume Form]\label{def:riemann-volume}
    If $\latsp$ is orientable then it has, up-to-sign, a canonical volume form $\volmeas{\latsp}$ called the \textit{Riemannian volume form}. It is the unique, up-to-sign, $d$-form such that for each $x\in\latsp$ and orthonormal basis $v_1,\dots, v_d \in \mathrm T_x\latsp$,
    \begin{equation}\label{eq:riemann-volume-measure-orthonormal}
        \abs{\volmeas{\latsp}(x)[ v_1,\dots,  v_d]} = 1.
    \end{equation}
    Without assuming orientability, $\latsp$ still admits, up-to-sign, a canonical volume form in each coordinate patch. Thus the absolute value of the Riemannian volume form is still well-defined globally. Abusing notation, we write it $|\omega_\latsp|$. Moreover, the induced \emph{Riemannian volume measure} $\omega_\latsp$ (see \cref{def:volume-measure}) is also well-defined in general.

    If $\mathbb Y$ and $\mathbb Z$ are Riemannian submanifolds of $\latsp$, $x\in \mathbb Y\cap \mathbb Z$ and $(T_x\mathbb Y)^\perp = T_x\mathbb Z$, then for $ v_1,\dots, v_n \in T_x\mathbb Y$, $ v_{n+1},\dots, v_d \in T_x\mathbb \latsp$, we have
    \begin{equation}\label{eq:diff-form-product}
        \abs{\volmeas{\latsp}(x)[ v_1, \dots,  v_d]} = \abs{\volmeas{\mathbb Y}(x)[ v_1, \dots,  v_n]\cdot \volmeas{\mathbb Z}(x)[\pi_{\mathbb Z}( v_{n+1}),\dots, \pi_{\mathbb Z}( v_d)]}
    \end{equation}
    where $\pi_{\mathbb Z}\colon T_x\latsp \to T_x\mathbb Z$ is the orthogonal projection.

    Moreover, given a chart $ \varphi\colon U\to V$, where $U\subset \latsp$, $V\subset \R^d$, for each $x\in U$ choose an orthonormal basis $v_1(x),\dots,v_d(x)\in\mathrm T_x\latsp$. Write $M(x)$ for the matrix representation of $\jacobian{\varphi}(x)$ with respect to the basis $v_1(x),\dots,v_d(x)$ and the standard basis of $\R^d$, $e_1,\dots,e_d$. Then the Riemannian volume measure pushed through the chart satisfies by \cref{eq:volume-measure}
    \begin{equation}\label{eq:riemann-volume-measure-pushforward}
        \begin{aligned}
            \dv{ \varphi_\star \volmeas{U}}{\lebesgue_V} {(z)} &= \abs{\varphi^{-1\star}\volmeas{U}(z)[ e_1,\dots,  e_m]} \\
            &= \abs{\volmeas{U}( \varphi\inv(z))[\jacobian{\varphi\inv}{(z)}[ e_1],\dots, \jacobian{\varphi\inv}{(z)}[ e_m]]}\\
            &= \abs{\det(M(\varphi\inv(z))\inv)} \\
            &= \abs{\det(M(\varphi\inv(z))^\top)\inv} \\
            &= \abs{\volmeas{U}( \varphi\inv(z))[\nabla\varphi(\varphi\inv(z))]}\inv.
        \end{aligned}
    \end{equation}
\end{definition}
\begin{definition}[Riemannian Tangent Volume]\label{def:riemann-tangent-volume}
    For each $x\in\latsp$, $\mathrm T_x\latsp$ has a canonical \emph{Riemannian tangent volume measure} $\lambda$, given by identifying an orthonormal basis of $\mathrm T_x\latsp$ with the unit vectors of $\R^d$ to obtain an isomorphism $\mathrm T_x\latsp\cong \R^d$, and taking the Lebesgue measure on $\R^d$.
\end{definition}
\begin{theorem}[Distance Function]\label{def:riemann-distance}
    $\latsp$ is a metric space when equipped with the \textnormal {distance function}
    \begin{equation*}
        d^\latsp(x,x') = \inf\{L(\gamma):\gamma\colon [0,1]\to\latsp \text{ a $C^k$ curve s.t. } \gamma(0)=x, \gamma(1)=x'\}
    \end{equation*}
    for $x,x'\in \latsp$, where
    \begin{equation*}
        L(\gamma) = \int_0^1\|\gamma'(t)\|\dd t.
    \end{equation*}
\end{theorem}
\begin{theorem}[Exponential Map]Assuming $k\geq 2$, for each $x\in\latsp$ there is a $C^{k-1}$ map $\exp_x^\latsp\colon V\to\latsp$ defined in a neighbourhood $V\subset \mathrm T_x\latsp$ of $0$ such that $\jacobian{\exp_x^\latsp}(0)\colon \mathrm T_x\latsp \to \mathrm T_x\latsp$ is the identity and $d^\latsp(x, \exp_x^\latsp v) = \|v\|$
for all $v\in V$. $\exp_x^\latsp$ is called the (Riemannian) \textnormal{exponential map} at $x\in \latsp$.
\end{theorem}

\section{Riemannian Volume of Small Balls}
In this appendix $\latsp$ is a $d$-dimensional $C^2$ Riemannian manifold equipped with a Riemannian metric $g$. Write the metric ball $B^\latsp_r(x) = \{x'\in\latsp : d(x,x')<r\}$. We study the asymptotics of the volume of balls $\omega_\latsp(B_r^\latsp(x))$ and $\omega_{\tilde \latsp}(B_r^\latsp(x))$ as $r\downarrow 0$, for $\tilde \latsp$ a submanifold of $\latsp$.
\begin{proposition}[Volume of Riemannian Balls]\label{prop:small-riemannian-balls}
    We have for $x\in\latsp$
    \begin{equation*}
        \omega_\latsp(B_r^\latsp(x)) = V_dr^d + o(r^d)
    \end{equation*}
    as $r\downarrow 0$, where $V_d$ is the volume of the $d$-dimensional unit Euclidean ball.
\end{proposition}
\begin{proof}
    $\exp_x^\latsp$ is a diffeomorphism from a neighbourhood $V\subset\mathrm T_x\latsp$ of $0$ to a neighbourhood $U\subset \latsp$ of $x$ in $\latsp$. So we can invert it on $U$, yielding a chart $\exp_x^{\latsp}|_V\inv\colon U \to V$. By \cref{eq:volume-measure}, for $r$ small enough such that $B_r^\latsp(0)\subset V$ and $v_1,\dots,v_d\in \mathrm T_x\latsp$ an orthonormal basis,
    \begin{align*}
        \omega_\latsp(B_r(x)) &= \int_{B_r(x)}\abs{\exp_x^{\latsp}|_V^\star\volmeas{\latsp}(v)[v_1,\dots,v_d]} \lambda(\dd v) \\
        &= \abs{\volmeas{\latsp}(x)[\jacobian{\exp_x^{\latsp}}(x) [v_1],\dots,\jacobian{\exp_x^{\latsp}}(x)[v_d]]}V_dr^d + o(r^d) \\
        &=\underbrace{\abs{\volmeas{\latsp}(x)[v_1,\dots,v_d]}}_{=1 \text{ (\cref{eq:riemann-volume-measure-orthonormal})}}V_dr^d + o(r^d) \\
        &= V_dr^d + o(r^d)
    \end{align*}
    as $r\downarrow 0$, where $\lambda$ is the Riemannian tangent volume measure (\cref{def:riemann-tangent-volume}).
\end{proof}
\begin{proposition}[Volume of Ambient Riemannian Balls]\label{prop:ambient-riemannian-balls}
    Let $\tilde {\mathbb X}$ be a $\tilde d$-dimensional $C^2$ Riemannian submanifold of $\latsp$. Then for $x\in \tilde {\mathbb X}$,
    \begin{equation*}
        \omega_{\tilde {\mathbb X}}(B_r^\latsp(x)) = V_{\tilde d}r^{\tilde d} + o(r^{\tilde d})
    \end{equation*}
    as $r\downarrow 0$, where $V_{\tilde d}$ is the volume of the $\tilde d$-dimensional unit Euclidean ball.
\end{proposition}
\begin{proof}
  $\exp_x^{\tilde {\mathbb X}}$ is a diffeomorphism from a neighbourhood $V\subset\mathrm T_x\tilde {\mathbb X}$ of $0$ to a neighbourhood $U\subset \tilde {\mathbb X}$ of $x$ in $\tilde {\mathbb X}$. So we can invert it on $U$.

  For $r>0$, it is clear that $B_r^{\tilde {\mathbb X}}(x) \subset B_r^\latsp(x)$. Now the Riemannian distance function $d^{\tilde {\mathbb X}}(x,\cdot)$ to $x$ on $\tilde {\mathbb X}$ is given for $x'\in U$ by
  \begin{equation*}
      d^{\tilde {\mathbb X}}(x,x') = \norm{\exp_x^{\tilde {\mathbb X}}|_{V}\inv(x')}.
  \end{equation*}
  Since the derivative of the exponential map at the origin of the tangent space is the identity, we have by Taylor's theorem, for $\exp_x^{\latsp}( v)\in U$
  \begin{align*}
      d^{\tilde {\mathbb X}}(x, \exp_x^{\latsp}( v)) = \norm{\exp_x^{\tilde {\mathbb X}}|_{V}\inv(\exp_x^{\latsp}(v))} = \norm{v + o(\|v\|)}= \|v\| + o(\|v\|)
  \end{align*}
    as $ v\to  0$. So $B_r^\latsp(x) \subset B_{r+o(r)}^{\tilde {\mathbb X}}(x)$ as $r\downarrow 0$. Now by \cref{prop:small-riemannian-balls}
  \begin{equation*}
      \volmeas{\tilde {\mathbb X}}(B_r^{\tilde {\mathbb X}}(x)) = V_{\tilde d}r^{\tilde d} + o(r^{\tilde d})
  \end{equation*}
  as $r\downarrow 0$. Hence
  \begin{align*}
      V_{\tilde d}r^{\tilde d} + o(r^{\tilde d}) &= \volmeas{\tilde {\mathbb X}}(B_r^{\tilde {\mathbb X}}(x)) \\
      &\leq\volmeas{\tilde {\mathbb X}}(B_r^\latsp(x)) \\
      &\leq \volmeas{\tilde {\mathbb X}}(B_{r+o(r)}^{\tilde {\mathbb X}}(x)) \\
      &= V_{\tilde d}(r+o(r))^{\tilde d} + o(r^{\tilde d}) \\
      &= V_{\tilde d}r^{\tilde d} + o(r^{\tilde d})
  \end{align*}
  as $r\downarrow 0$, i.e.
  \begin{equation*}
      \volmeas{\tilde {\mathbb X}}(B_r^\latsp(x)) = V_{\tilde d}r^{\tilde d} + o(r^{\tilde d})
  \end{equation*}
  as $r\downarrow 0$.
\end{proof}
The next proposition is of a similar flavour to \cref{prop:ambient-riemannian-balls}, but the ambient space is a normed space.
\begin{proposition}[Volume of Ambient Normed Balls]\label{prop:ambient-normed-balls}
    Let $\R^d$ be a normed space and $\tilde {\mathbb X}\subset
    \R^d$ a $\tilde d$-dimensional embedded $C^2$ manifold equipped with some Riemannian metric. Then for $x\in \tilde {\mathbb X}$,
    \begin{equation*}
        \omega_{\tilde {\mathbb X}}(B_r^{\R^d}(x)) = V_{\tilde d}(x)r^{\tilde d} + o(r^{\tilde d})
    \end{equation*}
    as $r\downarrow 0$ where $V_{\tilde d}(x)$ is the Riemannian tangent volume of $B_1^{\R^d}(0)\cap \mathrm T_x\tilde {\mathbb X}$, which in general depends on $x$, where we view $\mathrm T_x\tilde {\mathbb X}$ as a linear subspace of $\R^d$.
\end{proposition}
\begin{proof}
    We write $\|\cdot\|_{\R^d}$ for the ambient norm in $\R^d\supset \mathrm T_x\tilde {\mathbb X}$. $\exp_x^{\tilde {\mathbb X}}$ is a diffeomorphism from a neighbourhood $V\subset\mathrm T_x\tilde {\mathbb X}$ of $0$ to a neighbourhood $U\subset \tilde {\mathbb X}$ of $x$ in $\tilde {\mathbb X}$. So we can invert it on $U$. Then for $x'\in B_r^{\R^d}(0)\cap U \subset \tilde {\mathbb X}$ we have by Taylor's theorem
    \begin{equation*}
        \norm{\exp_x^{\tilde {\mathbb X}}|_{V}\inv(x')}_{\R^d} = \norm{x'-x + o(\|x'-x\|_{\R^d})}_{\R^d} = \|x'-x\|_{\R^d} + o(\|x'-x\|_{\R^d})
    \end{equation*}
    as $x'\to x$. Hence
    \begin{equation*}
        \exp_x^{\tilde {\mathbb X}}(B_{r-o(r)}^{\R^d}(0)\cap \mathrm T_x\tilde {\mathbb X}) \subset B_r^{\R^d}(x) \cap \tilde {\mathbb X} \subset \exp^{\tilde {\mathbb X}}_x(B_{r+o(r)}^{\R^d}(0)\cap \mathrm T_x\tilde {\mathbb X})
    \end{equation*}
    as $r\downarrow 0$. So similarly to the proof of \cref{prop:small-riemannian-balls}, $B_r^{\R^d}(x) \cap \tilde {\mathbb X} \subset \exp_x^{\tilde {\mathbb X}}(B_{r+o(r)}^{\R^d}(0)\cap \mathrm T_x\tilde {\mathbb X})$ implies
    \begin{align*}
        \omega_{\tilde {\mathbb X}}(B_r^{\R^d}(x)) &\leq \int_{B_{r+o(r)}^{\R^d}(0)\cap \mathrm T_x\tilde {\mathbb X}}\abs{\exp_x^{\latsp}|_V^\star\volmeas{\latsp}(v)[v_1,\dots,v_d]} \lambda(\dd v) \\
        &= V_{\tilde d}(x)(r + o(r))^{\tilde d} + o(r^{\tilde d}) \\
        &= V_{\tilde d}(x)r^{\tilde d} + o(r^{\tilde d})
    \end{align*}
    as $r\downarrow 0$. Similarly the inclusion $\exp^{\tilde {\mathbb X}}_x(B_{r-o(r)}^{\R^d}(0)\cap \mathrm T_x\tilde {\mathbb X}) \subset B_r^{\R^d}(x) \cap \tilde {\mathbb X}$ yields the reverse inequality
    \begin{equation*}
        \omega_{\tilde {\mathbb X}}(B_r^{\R^d}(x)) \geq  V_{\tilde d}(x)r^{\tilde d} + o(r^{\tilde d})
    \end{equation*}
    as $r\downarrow 0$, which concludes the proof.
\end{proof}

\section{Modes Lemmas}
\begin{lemma}[{\citealp[Lemma B.1(a)]{Ayanbayev2022GammaConvergence}}]
  \label{lem:exhaustive-om}
  An OM functional $\omfctl \colon \omdom \to \R$ is exhaustive if and only if there is $\latval_2 \in \omdom$ such that
  \begin{equation*}
    \lim_{r \downarrow 0} \frac{\prior(\oball{r}{\latval_1})}{\prior(\oball{r}{\latval_2})} = 0
  \end{equation*}
  for all $\latval_1 \in \latsp \setminus \omdom$.
\end{lemma}
\begin{lemma}
  \label{lem:omfctl-reweighted}
  Let $\latsp$ be locally compact and let $f \colon \latsp \to \R_{\ge 0}$ be continuous.
  Define the measure $\nu(\dd{x}) \defeq f(x) \prior(\dd{x})$.
  If $\omfctl \colon \omdom \to \R$ is an OM functional for $\prior$ such that $\omdom$ and $\operatorname{supp}^\circ(f) \defeq \set{\latval \in \latsp \where f(\latval) > 0}$ have a non-empty intersection, then $$\omfctl[\nu] \colon \omdom[\nu] \to \R, \latval \mapsto \omfctl(\latval) - \log f(\latval)$$ with $\omdom[\nu] \defeq \omdom \cap \operatorname{supp}^\circ(f)$ is an OM functional for $\nu$.
  If $\omfctl$ is exhaustive, then $\omfctl[\nu]$ is exhaustive as well.
\end{lemma}
\begin{proof}
  Since $f$ is continuous and $\latsp$ is locally compact, $f$ is locally uniformly continuous by the Heine-Cantor theorem.
  With this, the statement can be proven analogously to Lemma B.8 of \citet{Ayanbayev2022GammaConvergence}.
\end{proof}

\end{document}